\documentclass[12pt,a4paper]{article}
\usepackage[T1]{fontenc}
\usepackage[polish,english]{babel}

\usepackage{amsthm,amssymb,amsmath,mathtools}

\usepackage{tikz}
\usepackage{hhline}

\usepackage{hyperref}

\newtheoremstyle{styl}{}{}{}{}{\bf}{.}{ }{}
\theoremstyle{styl}
\newtheorem{tw}{Theorem}

\newtheorem{fa}[tw]{Fact}

\newtheorem{wn}[tw]{Corollary}

\theoremstyle{definition}
\newtheorem{df}[tw]{Definition}

\newtheorem{Przyklad}[tw]{Example}

\newcommand{\n}{\noindent}

\theoremstyle{remark}

\newcommand{\DS}{\mathcal{A}}
\newcommand{\I}{\mathcal{I}}
\newcommand{\J}{\mathcal{J}}

\newcommand{\h}{\varphi}

\newcommand{\Hom}{\mathsf{Hom}}
\newcommand{\PHom}{\mathsf{PHom}}

\newcommand{\dom}{\mathsf{dom}}

\renewcommand{\max}{\mathsf{max}}
\renewcommand{\min}{\mathsf{min}}

\newcommand{\algebra}[1]{\boldsymbol{#1}}
\newcommand{\A}{\algebra{A}}
\newcommand{\B}{\algebra{B}}
\newcommand{\C}{\algebra{C}}

\newcommand{\TO}{\rightarrow}
\newcommand{\OT}{\leftarrow}

\newcommand{\fun}{\longrightarrow}
\newcommand{\id}{\textsf{id}}

\usepackage{xcolor}

\title{Decomposition of Płonka sums\\ into direct systems}

\author{Krystyna Mruczek-Nasieniewska, Mateusz Klonowski}

\date{}

\begin{document}

\maketitle

\thispagestyle{empty}

\vspace{-1cm}

\begin{abstract}
The Płonka sum is an algebra determined using a structure called a direct system. By a direct system, we mean an indexed family of algebras with disjoint universes whose indexes form a join-semilattice s.t. if two indexes are in a partial order relation, then there is a homomorphism from the algebra of the first index to the algebra of the second index. The sum of the sets of the direct system determines the universe of the Płonka sum. Therefore, to speak about a Płonka sum, there must be a direct system on which this algebra is based. However, we can look at a Płonka sum the other way around, considering it as some given algebra, and ask whether it is possible to determine all direct systems of it systematically. 

In our paper, we will decompose the Płonka sum in such a way as to give a solution to the indicated problem. Moreover, our method works for any algebra of the kind considered in the article, and thus, we can determine if a given algebra is a Płonka sum. The proposed method is based on two concepts introduced in the paper: isolated algebra and Płonka homomorphism.      
\end{abstract}

\n\textbf{Key words}: decomposition, direct frame, direct system, isolated algebra, Płonka homomorphism, Płonka sums

\bigskip

\n\textbf{MSC}: 06F25, 08A30, 08A62 

\section*{Introduction}
\label{Section 0}

\hypertarget{Section 0}{In \cite{JP1967a} J. Płonka defined a certain algebraic structure, called in the literature a \textit{Płonka sum}.} 
The Płonka sum was constructed to analyze equational classes determined by certain regular identities, i.e., such identities where the sets of variables of the left and right sides are equal. Considering classes defined by regular identities is part of research on identities, in which the structure of terms creating them matters. Example of such identities are: normal, externally compatible, P-compatible, biregular, symmetric and many others (see for example \cite{EG1990,KMN2016,KMN2000,KMN2013,KMN2003,KMN2010,JP1967a,JP1967b,JPiAR1992,ARiJS2002}). In recent years, Płonka sums have been analyzed not only within universal algebra but also logic. For example, they are used as part of the semantic analysis of variable inclusion logic (see \cite{SBiTMiMPB2020,SBiFPiMPB2022,FPiMPBiDES2021}). This logic is an example of a formal system whose semantics take into account not only logical connections but also other dependencies, and relations between sentences (cf. \cite{RE1990,TJiFP2021,MK2021,SK1991}). 

To define a Płonka sum, we use a direct system, which is an indexed family of algebras with disjoint universes whose indexes form a join-semilattice s.t. if two indexes are in a partial order relation, then there is a homomorphism from the algebra of the first index to the algebra of the second index.
The sum of the sets of the direct system determines the universe of the Płonka sum. Therefore, to speak about a Płonka sum, there must be a direct system on which this algebra is based. However, we can look at a Płonka sum the other way around, considering it as some given algebra, and ask whether it is possible to determine all direct systems of it systematically.

Therefore, we can analyze the Płonka sum from two perspectives. From the first, the direct system of a given Płonka sum is determined and known. In other words, the first perspective employs a direct system to determine the Płonka sum. From the second, the Płonka sum is presented as a universe on which operations are defined. Thus, the second perspective acknowledges the Płonka sum as an algebra with a universe that can be expressed as a union of a family of sets forming a certain direct system. Of course, a direct system for a given Płonka sum must exist in both cases.
The difference is that in the first case, it is known, and in the second, it is not.
Usually, in the literature, the first perspective is adopted without asking how, given the algebra being the Płonka sum, all the direct systems that determine it can be identified. 

In our paper, we will present a method for determining the direct systems of a given Płonka sum by decomposing this sum so as to obtain an appropriate indexed family of algebras with disjoint universes forming semilattice and an appropriate set of homomorphisms. 
In other words, we will represent direct systems by structures we can receive by our method of Płonka sums decomposition. Such an approach fits into the investigations on structures' representation. A well-known example of a result related to this topic might be, for instance, Stone's representation theorem for Boolean algebras. Moreover, since the universe of the Płonka sum is a set-theoretic union, it seems natural to ask about the decomposition of the Płonka sum, which considers the division of its universe into disjoint sets.
Such a problem fits into the general investigations on complex structures' decomposition (or factorization). A well-known example of a result related to this topic might be, for instance, the fundamental theorem of arithmetic or the fundamental subdirect representation theorem of G. Birkhoff.

Note that in \cite{JP1967a}, Płonka presented a partition function determined on a universe of a given algebra. Using any of such functions, one can also decompose an algebra being a Płonka sum receiving its direct systems. However, one can also notice that finding a partition function starting with an algebra is a more difficult task than applying our method as we employ less complex notions.     

There are some other constructions known in the literature that propose some decompositions of Płonka sums (see \cite{HLiRPiCP1972,JP1989}). 
The first of the decompositions was presented in \cite{HLiRPiCP1972}. It indicates (with certain assumptions) subdirectly irreducible algebras in the class of all Płonka sums. It has been proven that they are either subdirect irreducible algebras in the initial class or are created from them after adding the so-called absorbing element.
The second of the decompositions was presented in \cite{{JP1989}}. This work shows that each Płonka sum can be represented as a subdirect product of two Płonka sums obtained from the initial sum, the so-called sum with zero and sum with unit. 

However, these decompositions are made assuming we know the direct system determining the initial Płonka sum. As mentioned in our paper, we will decompose algebras being Płonka sums without knowing their direct systems.
Our method works for any algebra of the kind considered in the article, and thus, we can determine if a given algebra is a Płonka sum.

The article consists of five sections. In Section \ref{Section 1}, we present some preliminary notions used in the paper. In Section \ref{Section 2}, we introduce the definition of a Płonka sum. In Section \ref{Section 3}, we introduce key concepts used in our decomposition method: isolated algebras and Płonka homomorphism. In Section \ref{Section 4}, we present the main result of the paper --- the decomposition theorem. Section \ref{Section 5}, contains some exemplary illustrations of the decomposition method we propose. The \hyperlink{Appendix}{Appendix} contains examples of algebra characterized by some definitions we present. In some cases, the algebras discussed in the \hyperlink{Appendix}{Appendix} allow us to better formulate various problems that interest us in this article.  

\section{Preliminaries}
\label{Section 1}

\hypertarget{Section 1}{Let $\tau\colon F\rightarrow N$ be a type of algebras where $F$  is a set of fundamental operation symbols and $N$ is the set of all positive integers.}  
For an algebra $\A=\langle A;F_{1}^{\A},\ldots,F_{n}^{\A}\rangle$, by $ar(k)$ (where $k \leq n$) we denote the arity of operation $F_{k}^{\A}$. We will denote the restriction of the function $F_{k}^{\A}$ to the set $B^{ar(k)}\subseteq A^{ar(k)}$ by $F_{k}^{\A}\! \!\upharpoonright_{B}$. We will say that two algebras are \textit{of the same type} (\textit{similar}) if they have the same number of operations with the same arity. In the sequel of the paper, we will analyze algebras having at least one operation of arity not smaller than $2$. As we will find out, Płonka sums of algebras in which all operations are unary are not interesting.

One of the main concepts used in the paper is the concept of join-semilattice. We will sometimes consider this structure as a pair $\langle A,R\rangle$, where $R$ is a partial order on $A$ and for any $x,y\in A$, there is the least upper bound of the set $\{x,y\}$, and sometimes as a structure with one two-argument the least upper bound operation $\vee$. 

Let $\A$ and $\B$ be similar algebras. By $\dom(\A)$ we will denote the universe of $\A$, and by $\Hom(\A,\B)$ we denote the set of all homomorphisms from $\A$ to $\B$. 

Let $\B,\C$ be subalgebras of $\A$. For any $\varphi\in\Hom(\B,\C)$, we define function $f_{\varphi}\colon A\fun A$ in the following way:
\[
f_{\varphi}(a)=
\begin{cases}
\varphi(a),\quad &\text{ if }a\in B\\
a,\quad &\text{ if }a\notin B.
\end{cases}
\]

By $\max$ and $\min$, we will denote the greatest element and the least element of a given set, respectively, with respect to a given order if such elements exist.

\section{Płonka Sum}
\label{Section 2}

We now present the Płonka sum formally.
We start with a definition of a direct system.

\begin{df}\label{def:system-prosty}
A \textit{direct system} is a triple $\DS=\langle\I,\{\A_{i}\}_{i\in I},\allowbreak\{\h_{i,j}\}_{i\leqslant j;i,j\in I}\rangle$ s.t.:
\begin{itemize}
\item[a.] $\I=\langle I,\leqslant\rangle$ is a join-semilattice,
\item[b.] $\{\A_{i}\}_{i\in I}$ a non-empty family of algebras of the same type and with disjoint universes,
\item[c.] $\{\h_{i,j}\}_{i,j\in I}$ is a family of homomorphism s.t. for all $i,j\in I$:
\begin{itemize}
    \item if $i\leqslant j$, then there is a homomorphism s.t. $\h_{i,j}:A_{i}\fun A_{j}$,
    \item if $i\leqslant j\leqslant k$, then $\h_{j,k}\circ\h_{i,j}=\h_{i,k}$,
    \item $\varphi_{i,i}=\id_{A_{i}}$.\footnote{For simplicity of notation, we will omit identity homomorphism in our considerations --- we will assume by default that $\varphi_{i,i}=\id_{A_{i}}$.}
\end{itemize}
\end{itemize}
Additionally, we adopt the following notions:
\begin{itemize}
\item a direct system $\DS$ is called a \textit{non-trivial} one iff there are $i,j\in I$ s.t. $i\neq j$,
\item an ordered pair $\langle\I,\{\A_{i}\}_{i\in I}\rangle$ s.t. the conditions a and b are satisfied we call a \textit{direct frame}.
\end{itemize}
\end{df}

Let us now introduce the definition of the Płonka sum.

\begin{df}\label{def:suma-Plonki}
Let $\DS=\langle\I,\{\A_{i}\}_{i\in I},\allowbreak\{\h_{i,j}\}_{i\leqslant j; i,j\in I}\rangle$ be a direct system.
A \textit{Płonka sum} (in symb.: $S(\DS)$) is an algebra $\langle A;F_{1}^{\A},\ldots F_{n}^{\A}\rangle$ similar to algebras of family $\{\A_{i}\}_{i\in I}$ s.t.:
\begin{itemize}
\item $A=\bigcup_{i\in I}A_{i}$,
\item for any $k\leq n$, for all $x_{1},\ldots,x_{ar(k)}\in A$:
\[
\notag
F_{k}^{\A}(a_{1},\ldots,a_{ar(k)}):=F_{k}^{\A_{m}}(\varphi_{i_{1},m}(a_{1}),\ldots,\varphi_{i_{ar(k)},m}(a_{ar(k)})),
\]
where for any $j\leq ar(k)$, $a_{j}\in A_{i_{j}}$ and $m=i_{1}\vee\ldots\vee i_{ar(k)}$.
\end{itemize}
\end{df}

\noindent In the \hyperlink{Appendix}{Appendix}, we present several examples of Płonka sums, which we will refer to in the sequel of our paper.

Having introduced the definitions of the direct system and the Płonka sum, we can try to compare our paper's main goal with the fundamental arithmetic theorem, Birkhoff's subdirect representation theorem, and Stone's representation theorem for Boolean algebras.  
Examining the Płonka sum by considering its direct systems can be compared to considering a given number as the product of its prime factors. In turn, the question about direct systems that determine a given Płonka sum can be compared to the problem of determining the prime factors of a given number. 
We can also make a comparison with Birkhoff's theorem. In the case of decomposing an algebra into subdirect products, the aim is to obtain algebras that are less complex than the initial algebra and which, when properly combined, allow us to obtain the initial algebra. Our paper aims to decompose an algebraic structure into `smaller' algebras so that we receive the former by union of the latter.
On the other hand, Stone's theorem can also be recalled as a good reference point for our research. In our article, we will show that every direct system is isomorphic to a certain system of isolated algebras --- we define the notion system of isolated algebras in the sequel.

In the \hyperlink{Introduction}{Introduction}, we also mentioned the notion of partition function. Let us also introduce its formal definition. 

\begin{df}\label{def:P-funkcja}
Let $\A=\langle\A;F_{1}^{\A},\ldots,F_{n}^{\A}\rangle$ be an algebra. A \textit{partition function of $\A$} is a function $f\colon A\times A\fun A$ s.t. the following conditions are satisfied:
\begin{enumerate}
\item[P1.] $f(f(a,b),c)=f(a,f(b,c))$,
\item[P2.] $f(a,a)=a$,
\item[P3.] $f(a,f(b,c))=f(a,f(c,b))$,
\item[P4.] $f(F_{k}^{\A}(a_{1},\ldots, a_{ar(k)}),b)=F_{k}^{\A}(f(a_{1},b),\ldots,f(a_{ar(k)},b))$,
\item[P5.] $f(b,F_{k}^{\A}(a_{1},\ldots, a_{ar(k)}))=f(b,F_{k}^{\A}(f(b,a_{1}),\ldots,f(b,a_{ar(k)})))$,
\item[P6.] $f(F_{k}^{\A}(a_{1},\ldots, a_{ar(k)}),a_{l})=F_{k}^{\A}(a_{1},\ldots, a_{ar(k)})$, where $1\leq l\leq ar(k)$,
\item[P7.] $f(a,F_{k}^{\A}(a,\ldots, a))=a$.
\end{enumerate}
\end{df}

\n In \cite{JP1967a}, Płonka, by means of a partition function, defines an equivalence relation on the universe of a given algebra. The equivalence classes are subalgebras of the given algebra and form a join-semilattice. Moreover, he shows how to define, by means of the partition function, the set of homomorphisms satisfying the conditions presented in Definition \ref{def:system-prosty}.c. Therefore, if one can define a partition function, one can also find a direct system. Płonka also shows that if we have a direct system, by means of homomorphisms, we can define a partition function. We have the following corollary:

\begin{wn}\label{wn:sum-Plonki-wtw-p-function}
Algebra $\A$ is a Płonka sum iff there is a partition function of  $\A$.     
\end{wn}

Although we can use a partition function to decompose a Płonka sum, we don't know how to find, in a systematic way, a function that satisfies conditions P1--P7. The conditions P1--P7 are not easy to satisfy, and connecting them straightforwardly with conditions assumed for the direct system isn't easy either. 
We think that we can offer an easier-to-be-applied and more straightforwardly aimed at receiving direct systems method of decomposition.

Since our goal is to indicate the direct system for an algebra being the Płonka sum, we have to determine the following components: 
\begin{itemize}
\item[$\blacktriangleright$] an indexed family of algebras with disjoint universes where the set of indexes form a join-semilattice;
\item[$\blacktriangleright$] a set of homomorphisms satisfying given conditions.
\end{itemize}
In the next section, we will reconstruct both the family of algebras and the set of homomorphisms. Then, we will show that the obtained system will, in fact, define the initial algebra which will mean that the goal of our article has been achieved. 

\section{Isolated Algebras and Płonka Homomorphism}
\label{Section 3}

Let us now introduce the notion of an isolated algebra and a Płonka homomorphism --- the key notions of our decomposition method of a Płonka sum. 

\subsection{Isolated Algebras}

\begin{df}\label{def:algebra-splitujaca}
Let $\A=\langle A;F_{1}^{\A},\ldots,F_{n}^{\A}\rangle$ and $\B=\langle B;F_{1}^{\B},\ldots,F_{n}^{\B}\rangle$ be algebras of the same type. Algebra $\B$ is an \emph{algebra isolated wrt $\A$} (an \textit{isolated algebra}) iff the following conditions are satisfied: 
\begin{itemize}\itemsep=5pt 
\item $\B$ is a subalgebra of $\A$, 
\item for every $m\leq n$, for all $a_1,\dots,a_{ar(m)}\in A$, if there is $i\leq ar(m)$ s.t. $a_i\not\in B$, then $F_{m}^{\A}(a_1,\dots,a_{ar(k)})\not\in B$.
\end{itemize}
\end{df}

\n The examples of algebras isolated wrt a Płonka sum were presented in Example \ref{Przyklad-Spl} (see \hyperlink{Appendix}{Appendix}).

As can be easily seen, in the case of algebras of type (1), each of their subalgebras will be isolated for the obvious reason. As we stated above, we will focus on algebras that have at least one operation of arity greater than $1$. Moreover, we will consider algebras without constants.

Additionally, in the article, we will only consider algebras with respect to which there are finitely many isolated algebras. This does not mean, however, that the algebras we are considering, in particular the Płonka sums, must be finite or that algebras isolated wrt them are finite. In Example \ref{Przyklad-Nieskonczona-SP} (see \hyperlink{Appendix}{Appendix}), an infinite algebra (Płonka sum) is specified, with respect to which we can determine finitely many infinite isolated algebras.

Note that in Example \ref{Przyklad-Nieskonczona-SP} (see \hyperlink{Appendix}{Appendix}), we indicated for the Płonka sum its direct system consisting of two algebras. The first one, with the smaller index, is an isolated algebra. The second one, with the greater index, is not isolated. But the union of the first and the second is an isolated algebra. We can generalize this observation. Namely, if we consider a union of the universe of a given algebra and the universes of all algebras with the smaller indexes, we will obtain an isolated algebra.

Let $\DS=\langle\I,\{\A_{i}\}_{i\in I},\{\h_{i,j}\}_{i\leqslant j;i,j\in I}\rangle$ be a direct system. For every $i\in I$, we put: 
\[
\label{def:Spl+}\tag{$\dagger$}A_{i}^{+}:=\bigcup_{j\leqslant i}A_{j}.
\]

\n We can observe that for every $i\in I$, $A_{i}^{+}\neq\emptyset$. Moreover, we can prove that for every $i\in I$, $A_{i}^{+}$ is closed under operations of $S(\DS)$. 

\begin{fa}\label{fa:DS->AlgSpl}
Let $\DS=\langle\I,\{\A_{i}\}_{i\in I},\{\h_{i,j}\}_{i\leqslant j;i,j\in I}\rangle$ be a direct system. 
Then, $A^{+}_{i}$ is closed under operations of $S(\DS)$. 
\end{fa}

\begin{proof}
Let $i\in I$, $m\leq n$ and $a_{1},\ldots, a_{ar(m)}\in A_{i}^{+}$. 
By Definition \ref{def:suma-Plonki}, $F^{S(\DS)}_{m}(a_{1},\allowbreak\ldots,\allowbreak a_{ar(m)})=F^{\A_{l}}_{m}(\h_{i_{1},l}(a_{1}),\ldots,\allowbreak\h_{i_{ar(m)},l}(\allowbreak a_{ar(m)}))$, where $l=i_{1}\vee\ldots\vee i_{ar(m)}$ and for every $j\leq ar(m)$, $a_{j}\in A_{i_{j}}$. By \eqref{def:Spl+} and Definition \ref{def:system-prosty}, $i_{1},\ldots, i_{ar(m)}\leqslant i$. Thus, $l\leqslant i$. Therefore, $A_{l}\subseteq A_{i}^{+}$, so $F^{S(\DS)}_{m}(a_{1},\ldots,a_{ar(m)})\in A_{i}^{+}$.
\end{proof}

\n Note that the following properties hold for any algebra $\A^{+}_{i}$ obtained from the algebra $\A_{i}$.

\begin{fa}\label{fa:DS->Spl}
Let $\DS=\langle\I,\{\A_{i}\}_{i\in I},\{\h_{i,j}\}_{i\leqslant j;i,j\in I}\rangle$ be a direct system. Then: 
\begin{enumerate}\itemsep=0pt
\item for all $i\in I$, $\A^{+}_{i}$ is an algebra isolated wrt $S(\DS)$,
\item for all $i,j\in I$, $A_{i}^{+}\subseteq A_{j}^{+}$ iff $i\leqslant j$,
\item there is $i\in I$ s.t. $\A_{i}^{+}=S(\DS)$,
\item for all $i,j\in I$, $A_{i}\cap A_{j}^{+}\neq\emptyset$ iff $i\leqslant j$,
\item for all $i,j\in I$, $A^{+}_{i}=A_{j}^{+}$ iff $i=j$,
\item for all $n\geq 2$, if $i_{1},\ldots, i_{n}\in I$ and $A_{i_{1}}^{+}\cap\ldots\cap A_{i_{n}}^{+}\neq\emptyset$, then there is $i\in I$ s.t. $A^{+}_{i}=A_{i_{1}}^{+}\cap\ldots\cap A_{i_{n}}^{+}$.
\end{enumerate}
\end{fa}

\begin{proof}
Ad 1. Let $i\in I$, $m\leq n$ and $a_{1},\ldots, a_{ar(m)}\in \bigcup_{j\in I}A_{j}$. We assume that for some $l\leq ar(m)$, $a_{l}\notin A_{i}^{+}$. Thus, for all $j\leqslant i$, $a_{l}\notin A_{j}$. Then, by Definition \ref{def:algebra-splitujaca}, for all $j\leqslant i$, $F_{m}^{S(\DS)}(a_{1},\ldots,a_{ar(m)})\notin A_{j}$. Therefore, by \eqref{def:Spl+}, $F_{m}^{S(\DS)}(a_{1},\ldots a_{ar(m)})\notin A_{i}^{+}$.

Ad 2. Let $i,j\in I$. ($\TO$) Suppose $A_{i}^{+}\subseteq A_{j}^{+}$. Thus, by \eqref{def:Spl+}, $A_{i}\subseteq A_{j}^{+}$, so for any $a\in A_{i}$, there is $k\leqslant j$ s.t. $a\in A_{k}$. By Definition \ref{def:system-prosty}, for all $k\in I$, $i=k$ or $A_{i}\cap A_{k}=\emptyset$. 
Therefore, $i\leqslant j$. ($\OT$) By \eqref{def:Spl+}, if $i\leqslant j$, then $A_{i}^{+}\subseteq A_{j}^{+}$.

Ad 3. By the assumption concerning the algebras we consider in the article and 1, $\I$ is a finite join-semilattice. 

Ad 4. Let $i,j\in I$. ($\TO$) Suppose that $A_{i}\cap A_{j}^{+}\neq\emptyset$. Let $a\in A_{i}\cap A_{j}^{+}$. Since $a\in A_{j}^{+}$, by \eqref{def:Spl+}, $a\in A_{j_{0}}$, where $j_{0}\leqslant j$. Thus, $A_{i}\cap A_{j_{0}}\neq\emptyset$. By Definition \ref{def:system-prosty}, ${i}={j_{0}}\leqslant j$. ($\OT$) By 3, if $i\leqslant j$, then $A_{i}\cap A_{j}^{+}\neq\emptyset$. 

Ad 5. Let $i,j\in I$. ($\TO$) Suppose that $A_{i}^{+}=A_{j}^{+}$. Thus, by \eqref{def:Spl+}, $\bigcup_{k\leqslant i}A_{k}=\bigcup_{k\leqslant j}A_{k}$. Hence, $A_{i}\subseteq A_{j}^{+}$ and $A_{j}\subseteq A_{i}^{+}$. By 4, $i=j$. Therefore, $A_{i}=A_{j}$. ($\OT$) Obvious. By Definition \ref{def:system-prosty}, if $i=j$, then $A_{i}=A_{j}$, so by \eqref{def:Spl+}, $A_{i}^{+}=A_{j}^{+}$. 

Ad 6. Let $B:=A_{i_{1}}^{+}\cap\ldots\cap A_{i_{n}}^{+}$ and $X:=\{A_{j}:j\in I\text{ and }A_{j}\cap B\neq\emptyset\}$. Note that $X$ is finite and non-empty. Thus, $X=\{A_{j_{1}},\ldots, A_{j_{k}}\}$, for some $k\geq 1$. Let $l:=j_{1}\vee\ldots\vee j_{k}$. By definition of $X$, $A_{i_{1}}^{+}\cap\ldots\cap A_{i_{n}}^{+}\subseteq A_{l}^{+}$. By 4, for all $o\leq k$, for all $p\leq n$ $j_{o}\leqslant i_{p}$. Thus, $A_{l}^{+}\subseteq A_{i_{1}}^{+}\cap\ldots\cap A_{i_{n}}^{+}$. 
\end{proof}

Let us now proceed to the analysis of the set of isolated algebras in terms of order and other set-theoretic properties.

The set of algebras isolated wrt a given algebra can be partially ordered using the inclusion relation. As the following fact shows, the relation of being an isolated algebra is also a relation of partial order. 

\begin{fa}\label{fa:Spl-subseteq}
Let $\A=\langle A; F_{1}^{\A},\ldots,F_{n}^{\A}\rangle$ be an algebra and $\{\A_{i}\}_{i\in I}$ be a set of algebras isolated wrt $\A$. Then, for all $i,j\in I$, $\A_{i}$ is isolated wrt $\A_{j}$ iff $A_{i}\subseteq A_{j}$.    
\end{fa}

\begin{proof}
($\TO$) By Definition \ref{def:algebra-splitujaca}.
($\OT$) Suppose  $A_{i}\subseteq A_{j}$. Since $\A_{i},\A_{j}$ are subalgebras of $\A$, $\A_{i}$ is a subalgebra of $\A_{j}$. Moreover, for any $m\leq n$, for all $a_{1},\ldots,a_{ar(m)}\in A_{j}$, there is $l\leq ar(m)$ s.t. $a_{l}\notin A_{i}$. Then, since $\A_{i}$ is isolated wrt $\A$ and $A_{j}\subseteq A$, we have $F_{m}^{\A}(a_{1},\ldots,a_{ar(m)})\notin A_{i}$. 
\end{proof}

Regarding set-theoretic operations defined on the set of isolated algebras, let us first note that the set difference of an algebra and the isolated algebra is either an empty set or some subalgebra of the initial algebra.

\begin{fa}\label{fa:odejmowanie-SP}
Let $\A=\langle A; F_{1}^{\A},\ldots,F_{n}^{\A}\rangle$ be an algebra and $\{\A_{i}\}_{i\in I}$ be a set of algebras isolated wrt $\A$. Then, for all $i,j\in I$, if $A_{i}\subsetneq A_{j}$, then $A_{j}\setminus A_{i}$ is closed under operations of $\A$. 
\end{fa}

\begin{proof}
By Definition \ref{def:algebra-splitujaca}.
\end{proof}

For our further considerations, the observation that a finite, non-trivial union of isolated incomparable algebras is not an algebra will also be important.

\begin{fa}\label{fa:suma-SP}
Let $\A=\langle A; F_{1}^{\A},\ldots,F_{n}^{\A}\rangle$ be an algebra and $\{\A_{i}\}_{i\in I}$ be a set of algebras isolated wrt $\A$. Then, if there is no $i\in I$ s.t. $\bigcup_{j\in (I\setminus \{i\})}A_{j}\subseteq A_{i}$ (there is no greatest element in $\{\A_{i}\}_{i\in I}$), then there is no subalgebra of $\A$ with the universe $\bigcup_{i\in I}A_{i}$.
\end{fa}

\begin{proof}
We use induction on the cardinality of $I$.

\textit{Base step}. For $I=\{i,j\}$. Let $a\in\A_{i}\setminus\A_{j}$ and $b\in\A_{j}\setminus\A_{i}$. Let $m\leq n$ and $c_{1},\ldots, c_{ar(m)}\in\{a,b\}$, where $c_{1}=a$ and $c_{2}=b$. Then, by Definition \ref{def:algebra-splitujaca}, $F_{m}^{\A}(c_{1},\ldots, c_{ar(m)})\not\in A_{i}\cup A_{j}$.

\textit{Inductive step}. Let $I=\{i_{1},\ldots,i_{l+1}\}$, where $l\geq 2$. 
Suppose there is no $i\in I$ s.t. $\bigcup_{j\in (I\setminus \{i\})}A_{j}\subseteq A_{i}$. 

Let $J=I\setminus\{i_{l+1}\}$. We consider two cases: 
\begin{enumerate}\itemsep=0pt
\item there is no $i\in J$ s.t. $\bigcup_{j\in (J\setminus \{i\})}A_{j}\subseteq A_{i}$,
\item there is $i\in J$ s.t. $\bigcup_{j\in (J\setminus \{i\})}A_{j}\subseteq A_{i}$.
\end{enumerate} 

Case 1. By the inductive hypothesis there is no algebra with universe $\bigcup_{i\in J}A_{i}$. Suppose that for $m\leq n$ and $a_{1},\ldots, a_{ar(m)}\allowbreak \in\bigcup_{i\in J}A_{i}$, $F_{m}^{\A}(a_{1},\allowbreak\ldots,\allowbreak a_{ar(m)})\notin\bigcup_{i\in J\setminus\{i_{l+1}\}}A_{i}$.
We consider two cases: 
\begin{enumerate}\itemsep=0pt
\item[1a.] $F_{m}^{\A}(a_{1},\ldots, a_{ar(m)})=a\in A_{i_{l+1}}$, 
\item[1b.] $F_{m}^{\A}(a_{1},\ldots, a_{ar(m)})=a\notin A_{i_{l+1}}$.
\end{enumerate}

Case 1a. Let $b\in\bigcup_{i\in J}A_{i}$ and $b\notin A_{i_{l+1}}$. Let us consider $c_{1},\ldots, c_{ar(m)}\in \{a,b\}$, where $c_{1}=a$ and $c_{2}=b$. By Definition \ref{def:algebra-splitujaca}, $F_{m}^{\A}(c_{1},\ldots, c_{ar(m)})\not\in \bigcup_{i\in I}A_{i}$. Case 1b. Obvious.

Case 2. Suppose for $i_{0}\in J$, $\bigcup_{j\in (J\setminus \{i_{0}\})}A_{j}\subseteq A_{i_{0}}$. Let $a\in A_{i_{l+1}}$ and $a\notin A_{i_{0}}$. Thus, $a\notin \bigcup_{j\in (I\setminus \{i_{l+1}\})}A_{j}$. Moreover, $b\in\bigcup_{j\in (J\setminus \{i_{l+1}\})}A_{j}$ and $b\not\in A_{i_{l+1}}$. Let us consider $c_{1},\ldots, c_{ar(m)}\in \{a,b\}$, where $c_{1}=a$ and $c_{2}=b$. By Definition \ref{def:algebra-splitujaca}, $F_{m}^{\A}(c_{1},\ldots, c_{ar(m)})\not\in \bigcup_{i\in I}A_{i}$.
\end{proof}

In turn, the non-empty intersection of a finite number of isolated algebras is an isolated algebra.

\begin{fa}\label{fa:iloczyn-SP}
Let $\A=\langle A;F_{1}^{\A},\ldots,F_{n}^{\A}\rangle$ be an algebra and $\{\A_{i}\}_{i\in I}$ be a set of algebras isolated wrt $\A$. Then, if $B:=\bigcap_{i\in I}A_{i}\neq\emptyset$, then $\langle B,F^{\A}_{1}\!\!\upharpoonright_{B},\ldots,F^{\A}_{n}\!\!\upharpoonright_{B}\rangle$ is algebra isolated wrt $\A$.
\end{fa}

\begin{proof}
Suppose $B:=\bigcap_{i\in J}A_{i}\neq\emptyset$. Obviously, $B$ is closed under operations of $\A$.
Let $m\leq n$ and $a_{1},\ldots,a_{ar(m)}\in A$. Let us assume that for $l\leq ar(m)$, $a_{l}\not\in B$. Thus, there is $i\in I$ s.t. $a_{l}\notin A_{i}$. By Definition \ref{def:algebra-splitujaca}, $F_{m}^{\A}(a_{1},\ldots,a_{ar(m)})\not\in A_{i}$. Therefore, $F_{m}^{\A}(a_{1},\ldots,a_{ar(m)})\not\in B$
\end{proof}

By Fact \ref{fa:odejmowanie-SP} and Fact \ref{fa:suma-SP} we derive the following corollary:

\begin{wn}\label{wn:odejmowanie+suma-SP}
Let $\A=\langle A;F_{1}^{\A},\ldots,F_{n}^{\A}\rangle$ be an algebra and $\{\A_{i}\}_{i\in I}$ be a set of algebras isolated wrt $\A$. Then, for all $i\in I$, $A_{i}\setminus\bigcup\{A_{j}\subseteq A:j\in I\text{ and }A_{j}\subsetneq A_{i}\}\neq\emptyset$ and $A_{i}\setminus\bigcup\{A_{j}\subseteq A:j\in I\text{ and }A_{j}\subsetneq A_{i}\}$ is closed under the operations of $\A$.
\end{wn}

\begin{proof}
Let $i\in I$. 
Let us consider the maximal elements of $\{A_{j}\subseteq A:j\in I\text{ and }A_{j}\subsetneq A_{i}\}$ and denote them by $A_{j_{1}},\ldots,A_{j_{k}}$. 
If $k=1$, then we use Fact \ref{fa:odejmowanie-SP}.
If $k\neq 1$, then, by Fact \ref{fa:suma-SP}, $A_{i}\setminus(A_{j_{1}}\cup\ldots\cup A_{j_{k}})\neq\emptyset$ and, by Fact \ref{fa:odejmowanie-SP}, $A_{i}\setminus(A_{j_{1}}\cup\ldots\cup A_{j_{k}})$ is closed under operations of $\A$.     
\end{proof}

We can now show that the set of all isolated algebras forms a join-semilattice.

\begin{fa}\label{fa:spl.pol-krata}
Let $\A=\langle A; F_{1}^{\A},\ldots,F_{n}^{\A}\rangle$ be an algebra and $\{\A_{i}\}_{i\in I}$ be the set of all algebras isolated wrt $\A$. Then, $\langle \{\A_{i}\}_{i\in I},\allowbreak\subseteq\rangle$ is a join-semilattice.
\end{fa}

\begin{proof}
For all $i,j\in I$, we put: $\dom(\A_{i}\vee\A_{j})=\bigcap\{A_{k}\subseteq A:k\in I\text{ and }A_{i}\cup A_{j}\subseteq A_{k}\}$ and for every $m\leq n$,
$F^{\A_{i}\vee\A_{j}}_{m}:=F^{\A}_{m}\!\upharpoonright_{\dom(\A_{i}\vee\A_{j})^{ar(m)}}$.
Obviously $\A_{i}\vee \A_{j}$ is an algebra. By Fact \ref{fa:iloczyn-SP}, $\A_{i}\vee \A_{j}$ is isolated wrt $\A$.
\end{proof}

Given the family $\{\A_{i}\}_{i\in I}$ of algebras isolated wrt $\A$ we can define a poset on the set of indexes $I$. Namely, $\I:=\langle I,\leqslant\rangle$, where for any $i,j\in I$: 
\[
i\leqslant j \text{ iff }A_{i}\subseteq A_{j}.
\]
We assume that $i< j$ iff $i\leqslant j$ and $i\neq j$. Of course, if $\{\A_{i}\}_{i\in I}$ is a join-semilattice wrt $\subseteq$, then $\I$ is also a join-semilattice wrt $\leqslant$. 

Let us also note that since we consider finite sets of isolated algebras, the relation defined as follows will be non-empty:
\[
i\prec j\text{ iff }i< j\text{ and there is no $k\in I$ s.t. $i< k< j$},
\]
for any $i,j\in I$.

By Definition \ref{def:system-prosty}, the structure $\langle\I,\{\A_{i}\}_{i\in I}\rangle$, where $\{\A_{i}\}_{i\in I}$, is the set of all algebras isolated wrt a given Płonka sum $\A$, is of course not a direct frame of $\A$ --- family $\{\A_{i}\}_ {i\in I}$ is not composed of algebras with disjoint universes.
However, as we will show in sequel, for each direct frame $\langle J,\{\B_{j}\}_{j\in J} \rangle$ of a given Płonka sum, the structure $\langle J,\{\B_{j }^{+}\}_{j\in J} \rangle$ is isomorphic to some substructure of $\langle\I,\{\A_{i}\}_{i\in I}\rangle$. Now we will define the substructures that interest us.

\begin{df}\label{def:SAF}
Let $\A=\langle A; F_{1}^{\A},\ldots, F_{n}^{\A}\rangle$ be an algebra and $\{\A_{i}\}_{i\in I}$ be a set of all algebras isolated wrt $\A$. A \textit{frame of algebras isolated wrt $\A$} is an ordered pair $\langle \J,\{\A_{j}\}_{j\in J}\rangle$ s.t. the following conditions are satisfied:
\begin{enumerate}\itemsep=0pt
\item[F1.] $J\subseteq I$ and $\{\A_{j}\}_{j\in J}$ is a non-trivial join-semilattice wrt $\subseteq$,
\item[F2.] $\max I\in J$, 
\item[F3.] for any $j\in I$, if there is $m\in\mathbb{N}$ s.t. $j_{1},\allowbreak\ldots,\allowbreak j_{m}\in J$ and $A_{j}=A_{j_{1}}\cap\ldots\cap A_{j_{m}}$, then $j\in J$. 
\end{enumerate}
\end{df}

The indicated conditions F2 and F3 guarantee that the frame of isolated algebras has the following properties:

\begin{fa}\label{fa:SAF-wlasnosci}
Let $\A=\langle A;F_{1}^{\A},\ldots, F_{n}^{\A}\rangle$ be a Płonka sum and $\langle\I,\{\A_{i}\}_{i\in I}\rangle$ be a frame of algebras isolated wrt $\A$. Then:
\begin{enumerate}\itemsep=0pt
\item[1.] $\bigcup_{i\in I}A_{i}=A$,
\item[2.] for all $i,j\in I$, if $i\neq j$ and $i,j$ are minimal elements of $\I$, then $A_{i}\cap A_{j}=\emptyset$
\item[3.] for every $a\in A$, set $\{i\in I:a\in A_{i}\}$ is non-empty and contains the least element.
\end{enumerate}
\end{fa}

\begin{proof}
Ad 1. By F2, $\bigcup_{i\in I}A_{i}=A_{\max I}=A$.

Ad 2. Suppose that $i\neq j$ and $i,j$ are minimal elements of $\I$. Assume indirectly that $A_{i}\cap A_{j}\neq\emptyset$. 
By Fact \ref{fa:iloczyn-SP}, there is an algebra isolated wrt $\A$ with the universe $A_{i}\cap A_{j}$. Assume its index is $k$. By F3, $k\in I$. Then, by the definition of $\leqslant$, $k< i,j$.

Ad 3. Let $a\in A$. By Fact \ref{fa:iloczyn-SP}, $\bigcap\{A_{i}\subseteq A:i\in I\text{ and }a\in A_{i}\}$ is an algebra and, by the definition of $\leqslant$ and and Definition \ref{def:SAF}, its index is $\min\{i\in I:a\in A_{i}\}$.
\end{proof}

As we will see, frames of isolated algebras allow us to describe all direct systems of a given sum. For now, let us note that each frame of algebras isolated wrt a Płonka sum determines this Płonka sum.

\begin{fa}
Let $\A=\langle A;F_{1}^{\A},\ldots,\allowbreak F_{n}^{\A}\rangle$ be a Płonka sum and $\langle\I,\{\A_{i}\}_{i\in I}\rangle$ be a frame of algebras isolated wrt $\A$. Let $\B=\langle \bigcup_{i\in I}A_{i};F_{1}^{\B},\ldots,F_{n}^{\B}\rangle$ be an algebra s.t. for all $m\leq n$, for all $a_{1},\ldots,a_{ar(m)}\in A$:
\[
F^{\B}_{m}(a_{1},\ldots,a_{ar(m)}):=F_{m}^{\A_{j}}(a_{1},\ldots,a_{ar(m)}),
\]
where for all $l\leq ar(m)$, $i_{l}=\min\{i\in I:a_{l}\in A_{i}\}$ and $j=i_{1}\vee\ldots\vee i_{ar(m)}$. 
Then: 
\[
\A=\B.
\]
\end{fa}


We will now show that the algebras obtained according to the definition of \eqref{def:Spl+} from the direct system form an isomorphic structure with some frame of isolated algebras.

\begin{fa}\label{fa:DS=SAF}
Let $\DS=\langle \langle I,\leqslant\rangle,\{\A_{i}\}_{i\in I},\{\varphi_{i,j}\}_{i\leqslant j;i,j\in I}\rangle$ be a direct system and $\langle \langle J,\sqsubseteq\rangle,\{\B_{j}\}_{j\in J}\rangle$ be a frame of algebras isolated wrt $S(\DS)$, where $\{\B_{j}\}_{j\in J}$ is the set of all algebras isolated wrt $S(\DS)$. Then:
\begin{enumerate}\itemsep=0pt
    
\item $f:I\fun J$ s.t. $f(i):=j$, where $A_{i}^{+}=B_{j}$, is injection,

\item $\langle\langle f(I),\sqsubseteq\rangle,\{\B_{k}\}_{k\in f(I)}\rangle$ and $\langle \langle I,\leqslant\rangle,\{\A_{i}^{+}\}_{i\in I}\rangle$ are isomorphic, 

\item $\langle\langle f(I),\sqsubseteq\rangle,\{\B_{k}\}_{k\in f(I)}\rangle$ is a frame of algebras isolated wrt $S(\DS)$.
\end{enumerate}
\end{fa}

\begin{proof}
Ad 1. By Fact \ref{fa:DS->Spl}.1 and Definition \ref{def:system-prosty}, for any $i\in I$, there is exactly one $j\in J$ s.t. $A_{i}^{+}=B_{j}$. Due to the disjoint universes of algebras from $\{A_{i}\}_{i\in I}$, $f$ is an injection. 

Ad 2. By 1, function $g\colon I\fun f(I)$ s.t. $g(i):=f(i)$ is a bijection. Let $i,j\in I$. We have then: $i\leqslant j$, by Fact \ref{fa:DS->Spl}.3, iff $A^{+}_{i}\subseteq A^{+}_{j}$ iff $A_{f(i)}\subseteq A_{f(j)}$, by the definition of partial order on the set of indexes, iff $f(i)\sqsubseteq f(j)$. 

Ad 3. Conditions F1 and F2 are satisfied in an obvious way.    
By Fact \ref{fa:DS->Spl}.5 F3 is also satisfied. 
\end{proof}

Algebras from the frame of isolated algebras are not disjoint. In order to obtain structures composed of algebras with disjoint universes, we will consider the set difference of algebras from the frame and its proper subalgebras. 

Let $\A=\langle A; F_{1}^{\A},\ldots,F_{n}^{\A}\rangle$ be an algebra and $\langle \I,\{\A_{i}\}_{i\in I}\rangle$ be a frame of algebras isolated wrt $\A$. For any $i \in I$, we put: 
\[
\label{def:-A}\tag{$\ddagger$}A_{i}^{-}:= A_{i}\setminus\bigcup_{j<i}A_{j}.   
\]
Note that by Corollary \ref{wn:odejmowanie+suma-SP}, for any $i\in I$, $A_{i}^{-}\neq\emptyset$ and $A_{i}^{- }$ is closed under the operations of a given algebra $\A$. 
For any $i\in I$, an algebra of the form:
\[
\A_{i}^{-}=\langle A_{i}^{-}; F_{1}^{\A_{i}^{-}},\allowbreak \ldots,F_{n}^{\A_{i}^{-}}\rangle
\]
we call a \textit{complement algebra} (\textit{of $\A_{i}$}).

The concept of the family of complement algebras is not sufficient to decompose the Płonka sum. Two problems arise. The first is that we have not specified how we should determine the values of operations on arguments from different algebras of this family of complement algebras. The second one is related to the problem of the existence of homomorphisms (see Appendix, Example \ref{Przyklad4}) --- the definition of which is required by the definition of the Płonka sum. 
We will solve both problems in the next section

\subsection{Płonka Homomorphism}  
 
Płonka's construction requires the existence of homomorphisms from algebras of a given index to algebras of a `greater' index, where the algebras have disjoint universes. 
As shown in Example \ref{Przyklad4} (see \hyperlink{Appendix}{Appendix}), there is not always a homomorphism from a complement algebra to a complement algebra with a `greater' index. 
Moreover, not every homomorphism from a complement algebra to a complement algebra is a homomorphism suitable for determining the operation of a Płonka sum in the way specified in Definition \ref{def:suma-Plonki}. Such a situation is shown in Example \ref{Przyklad5} (see \hyperlink{Appendix}{Appendix}). 

Thus, we can conclude that having a given algebra that is a Płonka sum and trying to identify the respective homomorphisms between complement algebras, we distinguish three possibilities. The first is when homomorphisms do not exist. Such a situation is presented in Example \ref{Przyklad4}. The second is when homomorphisms exist, but they do not satisfy the condition given in Definition \ref{def:suma-Plonki}. Such a situation is described in Example \ref{Przyklad5}. Finally, the third, when homomorphisms exist and they allow us to determine the operations of the Płonka sum according to the mentioned condition. Such a situation has been presented in Examples \ref{Przyklad4} and \ref{Przyklad5} (also in Examples \ref{Przyklad1} and \ref{Przyklad2}).

Below, we introduce the general notion of a Płonka homomorphism, which will allow us to identify all homomorphisms of the third possibility. 

\begin{df}\label{def:P-homomorfizm}
Let $\A=\langle A; F_{1}^{\A},\ldots,F_{n}^{\A}\rangle$ be an algebra, a family $\{\A_{i}\}_{i\in I}$ of subalgebras of $\A$ indexed by a set $I$ 
partially ordered  by $\leqslant$, 
and let $i,j\in I$, where $i\leqslant j$. 
A function $\varphi\in\Hom(\A_{i},\A_{j})$ is a  \emph{Płonka homomorphism} (\emph{P-homomorphism}) \emph{wrt $\A$} iff one of the following cases holds:
\begin{itemize}
\item $i=j$ and $\h$ is an identity homomorphism, 
\item $i<j$ and for any $m\leq n$ and $a_{1},\ldots,a_{ar(m)}\in A$:
\[
F^{\A}_{m}(a_{1},\ldots,\allowbreak a_{ar(m)})=F^{\A}_{m}(f_{\varphi}(a_{1}),\ldots,\allowbreak f_{\varphi}(a_{ar(m)})),
\]
if $F_{m}^{\A}(a_{1},\ldots, a_{ar(m)})\in \bigcup_{j\leqslant k}A_{k}$.
\end{itemize}
\end{df}

\n The set of all P-homomorphisms wrt $\A$ between algebras $\A_{i}$ and $\A_{j}$ is denoted by $\PHom^{\A}(\A_{i},\A_{j})$. 

Before we proceed with further analysis of a P-homomorphism, let us make some observations about the antecedent of the defining component of P-homomorphism in the context of complement algebras. 
\begin{fa}\label{fa:warunki}
Let $\A=\langle A; F_{1}^{\A},\ldots,F_{m}^{\A}\rangle$ be an algebra and let $\langle \I,\{\A_{i}\}_{i\in I}\rangle$ be a frame of algebras isolated wrt $\A$. Then, for all $m\leq n$, for all $a_{1},\ldots,a_{ar(m)}\in A$, if there is $l\leq ar(m)$ s.t. $a_{l}\in\bigcup_{i\leqslant j}A_{j}^{-}$, then $F_{m}^{\A}(a_{1},\ldots, a_{ar(m)})\in \bigcup_{i\leqslant j}A_{j}^{-}$.
\end{fa}

\begin{proof}
Let $m\leq n$ and $a_{1},\ldots,a_{ar(m)}\in A$. Assume that for $l_{0}\leq ar(m)$ we have $a_{l_{0}}\in\bigcup_{i\leqslant j}A_{j}^{-}$. Let $a_{l_{0}}\in A_{j_{0}}^{-}$, where $i\leqslant j_{0}$. Thus, by \eqref{def:-A}, if $j<i$, then $a_{l_{0}}\notin A_{j}$. Similarly, if $j\nleqslant i\nleqslant j$, then $a_{l_{0}}\notin A_{j}$. Assume that $j_{1}\nleqslant i\nleqslant j_{1}$ and $a_{l_{0}}\in A_{j_{1}}$. Hence, $j_{0}\nleqslant j_{1}$, $j_{1}\nleqslant j_{0}$ and $a_{l_{0}}\in A_{j_{0}}\cap A_{j_{1}}$. Since $a_{l_{0}}\in A_{j_{0}}\cap A_{j_{1}}$, so by Fact  \ref{fa:iloczyn-SP}, we have $j_{2}\in I$ s.t. $A_{j_{2}}=A_{j_{0}}\cap A_{j_{1}}$. However, $j_{2}<j_{0}$ and $a_{l_{0}}\in A_{j_{2}}$. Therefore, $a_{l_{0}}\not\in A_{j_{0}}^{-}$.

Assume that $F_{m}^{\A}(a_{1},\ldots, a_{ar(m)})\notin \bigcup_{i\leqslant j}A_{j}^{-}$. Then, by \eqref{def:-A}, for any $j\in I$, if $i\leqslant j$, then $F_{m}^{\A}(a_{1},\ldots, a_{ar(m)})\notin A_{j}^{-}$. However, there is  $j_{3}\in I$ s.t. $F_{m}^{\A}(a_{1},\ldots, a_{ar(m)})\in A_{j_{3}}^{-}$, hence also $F_{m}^{\A}(a_{1},\ldots, a_{ar(m)})\in A_{j_{3}}$. So, $j_{3}<i$ or  $j_{3}\nleqslant i\nleqslant j_{3}$. Then, $a_{l_{0}}\notin A_{j_{3}}$. Consequently, by Definition \ref{def:algebra-splitujaca}, $F_{m}^{\A}(a_{1},\ldots, a_{ar(m)})\notin A_{j_{3}}$.
\end{proof}

In general, the converse implication of the implication of Fact \ref{fa:warunki} doesn't have to hold.

Generalizing Example \ref{Przyklad-P-hom}, we can say that any homomorphism of a direct system is a P-homomorphism wrt the Płonka sum defined this direct system.

\begin{fa}\label{fa:varphi_ij->P-hom}
Let $\mathcal{A}=\langle\I,\{\A_{i}\}_{i\in I},\{\h_{i,j}\}_{i\leqslant j;i,j\in I}\rangle$ be a direct system. Then, for all $i,j\in I$, if $i\leqslant j$, then $\h_{i,j}$ is a P-homomorphism wrt $S(\DS)$.
\end{fa}

\begin{proof}
Let $i_{0},j_{0}\in I$ and $i_{0}\leqslant j_{0}$.
If $i_{0}=j_{0}$, then $\varphi_{i_{0},j_{0}}$ is a P-homomorphism. Assume that $i_{0}< j_{0}$. Let $m\leq n$ and $a_{1},\ldots,a_{ar(m)}\in \bigcup_{i\in I}A_{i}$. Assume that $F^{S(\DS)}_{m}(a_{1},\ldots, a_{ar(m)})\in\bigcup_{j_{0}\leqslant k}A_{k}$. 

Consider the case that there is no $l\leq ar(m)$ s.t. $a_{l}\in A_{i_{0}}$. Then, $F^{S(\DS)}_{m}(a_{1},\allowbreak\ldots,a_{ar(m)})\allowbreak=F^{S(\DS)}_{m}(f_{\h_{i_{0},j_{0}}}(a_{1}),\ldots,\allowbreak f_{\h_{i_{0},j_{0}}}(a_{ar(m)}))$.

Consider the case that there is $l\leq ar(m)$ s.t. $a_{l}\in A_{i_{0}}$. 
Without loss of generality we can assume that 
$l_{0}\neq ar(m)$, $a_{1},\ldots, a_{l_{0}}\in A_{i_{0}}$ and $a_{l_{0}+1},\ldots, \allowbreak a_{ar(m)}\notin A_{i_{0}}$.

By Definition \ref{def:suma-Plonki}, $F^{S(\DS)}_{m}(a_{1},\ldots,a_{l_{0}},a_{l_{0}+1},\ldots\allowbreak a_{ar(m)})=F^{\A_{k_{0}}}_{m}(\allowbreak\varphi_{i_{0},k_{0}}(\allowbreak a_{1}),\allowbreak\ldots,\allowbreak\varphi_{i_{0},k_{0}}(\allowbreak a_{l_{0}}),\varphi_{j_{l_{0}+1},k_{0}}(\allowbreak a_{l_{0}+1}),\allowbreak\ldots,\varphi_{j_{ar(m)},k_{0}}(a_{ar(m)}))$, where $i_{0}\vee j_{l_{0}+1}\vee\allowbreak\ldots\vee j_{ar(m)}=k_{0}$ and for any $l_{0}+1\leq l\leq ar(m)$, $a_{l}\in A_{j_{l}}$. 

By the initial assumption $F^{S(\DS)}_{m}(a_{1},\ldots, a_{ar(m)})\in\bigcup_{j_{0}\leqslant k}A_{k}$, so $j_{0}\leqslant k_{0}$. Besides, since $k_{0}=i_{0}\vee j_{l_{0}+1}\vee\ldots\vee j_{ar(m)}\leqslant j_{0}\vee j_{l_{0}+1}\vee\ldots\vee j_{ar(m)}$, then $k_{0}=j_{0}\vee j_{l_{0}+1}\vee\ldots\vee j_{ar(m)}$.

By Definition \ref{def:system-prosty}, for any $l\leq l_{0}$, $\varphi_{i_{0},k_{0}}(a_{l})=\varphi_{j_{0},k_{0}}(\varphi_{i_{0},j_{0}}(a_{l}))$. Thus, by Definition \ref{def:system-prosty}, we have that: 
\begin{flalign*}
F^{S(\DS)}_{m}(a_{1},\ldots,\allowbreak a_{ar(m)})
=\,&F^{\A_{k_{0}}}_{m}(\varphi_{i_{0},k_{0}}(a_{1}),\ldots,\varphi_{i_{0},k_{0}}(a_{l_{0}}),&\\
&\,\varphi_{j_{l_{0}+1},k_{0}}(a_{l_{0}+1}),\ldots,\varphi_{j_{ar(m)},k_{0}}(a_{ar(m)}))&\\
=\,&F^{\A_{k_{0}}}_{m}(\varphi_{j_{0},k_{0}}(\varphi_{i_{0},j_{0}}(a_{1})),\ldots,\varphi_{j_{0},k_{0}}(\varphi_{i_{0},j_{0}}(a_{l_{0}})),&\\
\,&\varphi_{j_{l_{0}+1},k_{0}}(a_{l_{0}+1}),\ldots,\varphi_{j_{ar(m)},k_{0}}(a_{ar(m)}))&\\
=\,&F^{S(\DS)}_{m}(\varphi_{i_{0},j_{0}}(a_{1}),\ldots,\varphi_{i_{0},j_{0}}(a_{l_{0}}),a_{l_{0}+1},\ldots,a_{ar(m)})&\\
=\,&F^{S(\DS)}_{m}(\allowbreak f_{\h_{i_{0},j_{0}}}(a_{1}),\ldots,f_{\h_{i_{0},j_{0}}}(a_{ar(m)})).&\qedhere    
\end{flalign*}
\end{proof}

Let us now turn to a discussion of the issue of the composition of P-homomorphisms (cf. the conditions for the composition of homomorphisms given in Definition \ref{def:system-prosty}). Let us first note that from a complement algebra to a complement algebra with a `greater' index there can be more than one P-homomorphism. Such a situation is shown in Example \ref{Przyk.:P-homomorfizmy} (see \hyperlink{Appendix}{Appendix}). 
Below we will show that the composition of any P-homomorphisms defined for a complement algebra is a P-homomorphism. Let us start with the following fact.

\begin{fa}\label{fa.pomocniczy:h_ij+h_jk=P-homomorfizm}
Let $\A=\langle A; F_{1}^{\A},\ldots, F_{n}^{\A}\rangle$ be an algebra, 
$\langle\I,\{\A_{i}\}_{i\in I}\rangle$ be a frame of algebras isolated wrt $\A$, $\varphi_{i,j}\in\PHom^{\A}(\A_{i}^{-},\A_{j}^{-})$ and $\varphi_{j,k}\in\PHom^{\A}(\A_{j}^{-},\A_{k}^{-})$, while $i,j, k\in I$ such that $i< j< k$. 
Then,
for any $m\leq n$ and $a_{1},\ldots,\allowbreak a_{ar(m)} \allowbreak\in\allowbreak A$, if $F^{\A}_{m}(a_{1},\ldots, a_{ar(m)})\in\bigcup_{k\leqslant p}A_{p}^{-}$, then 
$F^{\A}_{m}(f_{\h_{j,k}}\allowbreak(f_{\h_{i,j}}\allowbreak(a_{1})),\ldots,\allowbreak f_{\h_{j,k}}(\allowbreak f_{\h_{i,j}}\allowbreak(a_{ar(m)})))=F^{\A}_{m}(f_{\h_{j,k}\circ\,\h_{i,j}}(a_{1}),\ldots, f_{\h_{j,k}\circ\,\h_{i,j}}(a_{ar(m)}))$.
\end{fa}

\begin{proof}
Let $a_{1},\ldots,a_{ar(m)}\in A$ and assume $F^{\A}_{m}(a_{1},\ldots, a_{ar(m)})\in\bigcup_{k\leqslant p}A_{p}^{-}$.
Hence, $F^{\A}_{m}(a_{1},\ldots, a_{ar(m)})\in\bigcup_{j\leqslant p}A_{p}^{-}$.

We will show that there is $l\leq ar(m)$ s.t. $a_{l}\notin A_{i}^{-}$ and $a_{l}\notin A_{j}^{-}$. Assume that for any $l\leq ar(m)$, $a_{l}\in A^{-}_{i}$ or $a_{l}\in A^{-}_{j}$. Then, $F^{\A}_{m}(f_{\h_{i,j}}(a_{1}),\ldots, \allowbreak f_{\h_{i,j}}( a_{ar(m)}))\in A_{j}^{-}$.
By Definition \ref{def:P-homomorfizm}, $F^{\A}_{m}(a_{1},\ldots, a_{ar(m)})=F^{\A}_{m}(f_{\h_{i,j}}(a_{1}),\allowbreak\ldots, f_{\h_{i,j}}(a_{ar(m)}))$. Therefore, by \eqref{def:-A}, $F^{\A}_{m}(a_{1},\ldots, a_{ar(m)})\notin\bigcup_{k\leqslant p}A_{p}^{-}$.

Consider the following cases:
\begin{itemize}
    \item[a.] there is no $l\leq ar(m)$ s.t. $a_{l}\in A_{i}^{-}$ and there is no $l\leq ar(m)$ s.t. $a_{l}\in A_{j}^{-}$, 
    \item[b.] there is no $l\leq ar(m)$ s.t. $a_{l}\in A_{i}^{-}$ and there is $l\leq ar(m)$ s.t. $a_{l}\in A_{j}^{-}$, 
    \item[c.] there is $l\leq ar(m)$ s.t. $a_{l}\in A_{i}^{-}$ and there is no $l\leq ar(m)$ s.t. $a_{l}\in A_{j}^{-}$, 
    \item[d.] there is $l\leq ar(m)$ s.t. $a_{l}\in A_{i}^{-}$ and there is $l\leq ar(m)$ s.t. $a_{l}\in A_{j}^{-}$.
\end{itemize}

Case a. By the definition of the function $f_{\h}$ we get: 
\begin{flalign*}
&F^{\A}_{m}(f_{\h_{j,k}}(f_{\h_{i,j}}(a_{1})),\ldots,f_{\h_{j,k}}(f_{\h_{i,j}}(a_{ar(m)})))=&\\
&=F^{\A}_{m}(a_{1},\ldots,a_{ar(m)})&\\
&=F^{\A}_{m}(f_{\h_{j,k}}\circ\,f_{\h_{i,j}}(a_{1}),\ldots,f_{\h_{j,k}}\circ\,f_{\h_{i,j}}(a_{ar(m)})).&
\end{flalign*}

Case  b. Without loss of generality we assume that only $a_{1},\allowbreak\ldots, a_{l}\in A_{j}^{-}$, where $l\lneq ar(m)$. By the definition of the function $f_{\h}$ and by Definition \ref{def:P-homomorfizm} we have: 
\begin{flalign*}
&F^{\A}_{m}(f_{\h_{j,k}}(f_{\h_{i,j}}(a_{1})),\ldots,f_{\h_{j,k}}( f_{\h_{i,j}}(a_{ar(m)})))=&\\
&=F^{\A}_{m}(\h_{j,k}( a_{1}),\ldots,\h_{j,k}(a_{l}),a_{l+1}, \ldots, a_{ar(m)})&\\
&=F^{\A}_{m}(f_{\h_{j,k}}(a_{1}),\ldots,f_{\h_{j,k}}(a_{l}),f_{\h_{j,k}}(a_{l+1}), \ldots, f_{\h_{j,k}}(a_{ar(m)}))&\\
&=F^{\A}_{m}(a_{1},\ldots,a_{l},a_{l+1},\ldots,a_{ar(m)})&\\
&=F^{\A}_{m}(f_{\h_{j,k}\circ\,\h_{i,j}}(a_{1}),\ldots,f_{\h_{j,k}\circ\,\h_{i,j}}(a_{l}),f_{\h_{j,k}\circ\,\h_{i,j}}(a_{l+1}),\ldots,f_{\h_{j,k}\circ\,\h_{i,j}}(a_{ar(m)})).&
\end{flalign*}

Case  c. Without loss of generality, we assume that only 
$a_{1},\allowbreak\ldots, a_{l}\in A_{i}$, where $l\lneq ar(m)$. By the definition of the function $f_{\h}$ we get: 
\begin{flalign*}
&F^{\A}_{m}(f_{\h_{j,k}}(f_{\h_{i,j}}(a_{1})),\ldots,f_{\h_{j,k}}( f_{\h_{i,j}}(a_{ar(m)})))=&\\
&=F^{\A}_{m}(\h_{j,k}(\h_{i,j}(a_{1})),\ldots,\h_{j,k}(\h_{i,j}(a_{l})),a_{l+1}, \ldots, a_{ar(m)})&\\
&=F^{\A}_{m}(f_{\h_{j,k}\circ\,\h_{i,j}}(a_{1}),\ldots,f_{\h_{j,k}\circ\,\h_{i,j}}(a_{l}),f_{\h_{j,k}\circ\,\h_{i,j}}(a_{l+1}),\ldots,f_{\h_{j,k\circ\,\h_{i,j}}}(a_{ar(m)})).&
\end{flalign*}

Case  d. Without loss of generality, we assume that only
$a_{1},\allowbreak\ldots, a_{l}\in A_{j}^{-}$ and $a_{l+1},\ldots, a_{l+p}\in A_{i}^{-}$, where $l+p\lneq ar(m)$. Notice that $\h_{j,k}(\h_{i,j}(a_{l+1})),\ldots,\allowbreak\h_{j,k}(\h_{i,j}(a_{l+p}))\in\bigcup_{k\leqslant p}A^{-}_{p}$. By the definition of the function $f_{\h}$, Fact \ref{fa:warunki} and Definition \ref{def:P-homomorfizm}, we have: 
\begin{flalign*}
&\,F^{\A}_{m}(f_{\h_{j,k}}(f_{\h_{i,j}}(a_{1})),\ldots,f_{\h_{j,k}}(f_{\h_{i,j}}(a_{ar(m)}))) =&\\
&=F^{\A}_{m}(f_{\h_{j,k}}(a_{1}),\ldots,f_{\h_{j,k}}(a_{l}),\h_{j,k}(\h_{i,j}(a_{l+1})),\ldots,\h_{j,k}(\h_{i,j}(a_{l+p})),&\\
&\;\;\;\;\;f_{\h_{j,k}}(a_{l+p+1}),\ldots,f_{\h_{j,k}}(a_{ar(m)})) &\\
&=F^{\A}_{m}(f_{\h_{j,k}}(a_{1}),\ldots,f_{\h_{j,k}}(a_{l}),f_{\h_{j,k}}(\h_{j,k}(\h_{i,j}(a_{l+1}))),\ldots,f_{\h_{j,k}}(\h_{j,k}(\h_{i,j}(a_{l+p}))),&\\
&\;\;\;\;\;f_{\h_{j,k}}(a_{l+p+1}),\ldots,f_{\h_{j,k}}(a_{ar(m)})) &\\
&=F^{\A}_{m}(a_{1},\ldots,a_{l},\h_{j,k}(\h_{i,j}(a_{l+1})),\ldots,\h_{j,k}(\h_{i,j}(a_{l+p})),&\\
&\;\;\;\;\;a_{l+p+1},\ldots,a_{ar(m)}) &\\
&=F^{\A}_{m}(f_{\h_{j,k}\circ\,\h_{i,j}}(a_{1}),\ldots,f_{\h_{j,k}\circ\,\h_{i,j}}(a_{l}),f_{\h_{j,k}\circ\,\h_{i,j}}(a_{l+1}),\ldots,f_{\h_{j,k}\circ\,\h_{i,j}}(a_{l+p}),&\\
&\;\;\;\;\;f_{\h_{j,k}\circ\,\h_{i,j}}(a_{l+p+1}),\ldots,f_{\h_{j,k}\circ\,\h_{i,j}}(a_{ar(m)})).\hspace{5.9cm}\qedhere&
\end{flalign*}
\end{proof}

Now we can  easily prove the following fact concerning the composition of P-homomorphisms.
\begin{fa}\label{fa:h_ij+h_jk=P-homomorfizm}
Let $\A=\langle A; F_{1}^{\A},\ldots, F_{n}^{\A}\rangle$ be an algebra and  $\langle\I,\{\A_{i}\}_{i\in I}\rangle$ be a  frame of algebras isolated wrt $\A$. Then, for all $i,j,k\in I$, if $i\leqslant j\leqslant k$, $\varphi_{i,j}\in\PHom^{\A}(\A_{i}^{-},\A_{j}^{-})$ and $\varphi_{j,k}\in\PHom^{\A}(\A_{j}^{-},\A_{k}^{-})$, then $\varphi_{j,k}\circ\varphi_{i,j}\in\PHom^{\A}(\A_{i}^{-},\A_{k}^{-})$.
\end{fa}

\begin{proof}
Assume that $i<j<k$. In the other cases, the composition equals one of the given P-homomorphisms. 

Let $a_{1},\ldots,a_{ar(m)}\in A$. Assume that $F^{\A}_{m}(a_{1},\ldots, a_{ar(m)})\in\bigcup_{k\leqslant p}A_{p}^{-}$. Hence, $F^{\A}_{m}(a_{1},\ldots, a_{ar(m)})\in\bigcup_{j\leqslant p}A_{p}^{-}$. Be Definition \ref{def:P-homomorfizm}, $F^{\A}_{m}(f_{\h_{i,j}}(a_{1}),\allowbreak\ldots,\allowbreak f_{\h_{i,j}}(a_{ar(m)}))=F^{\A}_{m}(a_{1},\ldots,  a_{ar(m)})\in\bigcup_{k\leqslant p}A_{p}^{-}$.
Therefore, by Definition \ref{def:P-homomorfizm} and Fact \ref{fa.pomocniczy:h_ij+h_jk=P-homomorfizm} we obtain:
\begin{flalign*}
F^{\A}_{m}(a_{1},\ldots, a_{m})&=F^{\A}_{m}(f_{\h_{i,j}}(a_{1}),\ldots, f_{\h_{i,j}}(a_{m}))&\\
&=F^{\A}_{m}(f_{\h_{j,k}}(f_{\h_{i,j}}(a_{1})),\ldots, f_{\h_{j,k}}(f_{\h_{i,j}}(a_{m})))&\\
&=F^{\A}_{m}(f_{\h_{j,k}\circ\,\h_{i,j}}(a_{1}),\ldots, f_{\h_{j,k}\circ\,f_{\h_{i,j}}}(a_{m})).&\qedhere
\end{flalign*}
\end{proof}

In the set of all frames of isolated algebras, we have to distinguish those frames,
for which from any complement algebra to any complement
algebra of a `greater' index, there is a P-homomorphism. Moreover, such homomorphisms have to satisfy the conditions for the composition of homomorphisms given in Definition \ref{def:system-prosty}. 
To this aim, we introduce the notion of a sound set of P-homomorphisms.

\begin{df}\label{def:odpowiedni.P-homomorfizm}
Let $\A=\langle A; F_{1}^{\A},\ldots,F_{n}^{\A}\rangle$ be an algebra and 
$\langle\I,\{\A_{i}\}_{i\in I}\rangle$ 
be a frame of algebras  isolated wrt $\A$.
A \textit{sound set of P-homomorphisms} is a set $\{\h_{i,j}\}_{i\leqslant j;i,j\in I}$ s.t. for any $i,j\in I$:
\begin{itemize}
\item if $i=j$ or $i\prec j$, then $\varphi_{i,j}\in\PHom^{\A}(\A_{i}^{-},\A_{j}^{-})$,
\item if $i<j$ and $i\nprec j$, then the following conditions are satisfied: 
\begin{itemize}
\item $\varphi_{i,j}=\h_{i_{m},j}\circ\ldots\circ\h_{i,i_{1}}$, for some $i_{1},\ldots,i_{m}\in I$ s.t. $i\prec i_{1}\prec\ldots\prec i_{m}\prec j$,
\item for any $j_{1},\ldots,j_{o}\in I$, if  $i\prec j_{1}\prec\ldots\prec j_{o}\prec j$, then $\varphi_{i,j}=\h_{j_{o},j}\circ\ldots\circ\h_{i,j_{1}}$.
\end{itemize}
\end{itemize}
\end{df}




\section{Decomposition Theorem}
\label{Section 4}

Let us now move on to the last construction. 

\begin{df}\label{def:kanoniczny.system.prosty}
Let $\A=\langle A; F_{1}^{\A},\ldots,F_{n}^{\A}\rangle$ be an algebra. 
\begin{itemize}\itemsep=0pt
    \item A \emph{system of algebras isolated wrt $\A$} is a triple $\langle\I,\{\A_{i}^{-}\}_{i\in I},\allowbreak\{\h_{i,j}\}_{i\leqslant j;\allowbreak i,j\in I}\rangle$ s.t. $\langle\I,\{\A_{i}\}_{i\in I}\rangle$ is a frame of algebras isolated wrt $\A$ and $\{\h_{i,j}\allowbreak\}_{i\leqslant j;i,j\in I}$ is a sound set of P-homomorphisms. 
    \item We call a system of isolated algebras $\langle\I,\{\A_{i}^{-}\}_{i\in I},\allowbreak\{\h_{i,j}\}_{i\leqslant j;\allowbreak i,j\in I}\rangle$ a \textit{non-trivial} one iff there are $i,j\in I$ s.t. $i\neq j$.
\end{itemize}
\end{df}

Let us now show that the set of systems of isolated algebras 
is non-empty.

\begin{tw}\label{pierwsze twierdzenie o dekompozycji}
If  $\A=\langle A; F_{1}^{\A},\ldots,F_{n}^{\A}\rangle$ is Płonka sum, then there exists  system of algebras isolated wrt $\A$.
\end{tw}

\begin{proof}
Let $\DS=\langle \langle I,\leqslant\rangle,\{\A_{i}\}_{i\in I},\{\h_{i,j}\}_{i\leqslant j;i,j\in I}\rangle$ be a direct system s.t. $\A=S(\DS)$. Let $\langle \langle J,\sqsubseteq\rangle,\{\B_{j}\}_{j\in J}\rangle$ is a frame of algebras isolated wrt $\A$, where $\{\B_{j}\}_{j\in J}$ is the set of all isolated algebras wrt $\A$. 

By Fact \ref{fa:DS=SAF}, we have the injection $f\colon I\fun J $ defined by the formula $f(i):=j$, where $A^{+}_{i}=B_{j}$. We have also that $\langle \langle I,\leqslant\rangle,\{\A_{i}^{+}\}_{i\in I}\rangle$ and $\langle \langle f(I),\allowbreak \sqsubseteq\rangle,\allowbreak \{\B_{k}\}_{k\in f(I)}\rangle$ are isomorphic. Therefore, isomorphic are also $\langle \langle I,\leqslant\rangle,\{\A_{i}\}_{i\in I}\rangle$ and $\langle \langle f(I),\sqsubseteq\rangle,\{\B_{k}^{-}\}_{k\in f(I)}\rangle$.

For all $j_{1},j_{2}\in J$, for every $a\in B_{j_{1}}$ we put
$\psi_{j_{1},j_{2}}(a):=\varphi_{i_{1},i_{2}}(a)$, where $j_{1}=f(i_{1})$ and $j_{2}=f(i_{2})$. By Fact \ref{fa:varphi_ij->P-hom} and Definition \ref{def:system-prosty}, $\{\psi_{j,k}\}_{j\sqsubseteq k;j,k\in f(I)}$ is a sound set of P-homomorphisms.

So, $\langle \langle f(I),\sqsubseteq\rangle,\{\B_{k}\}_{k\in f(I)},\{\psi_{j,k}\}_{j\sqsubseteq k;j,k\in f(I)}\rangle$ is a  system of algebras isolated wrt $\A$.
\end{proof}

The following theorem shows that isolated algebras allow us to decompose the Płonka sum.

\begin{tw}[Decomposition theorem]\label{Twierdzenie o dekompozycji}
If $\A=\langle A; F_{1}^{\A},\ldots,F_{n}^{\A}\rangle$ is an algebra and $\DS=\langle\I,\allowbreak\{\A_{i}^{-}\}_{i\in I},\allowbreak\{\h_{i,j}\}_{i\leqslant j;i,j\in I}\rangle$ is a non-trivial system of algebras isolated wrt $\A$, then $\DS$ is a direct system and $\A=S(\DS)$.
\end{tw}

\begin{proof}
Obviously, $\DS$ is a non-trivial direct system.

Let $m\leq n$ and $a_{1},\ldots,a_{ar(m)}\in A$. Assume that for any  $l\leq ar(m)$, $a_{l}\in A^{-}_{i_{l}}$ and  $j=i_{1}\vee\ldots\vee i_{ar(m)}$. Thus, for any $l\leq ar(m)$, $\h_{i_{l},j}a_{l}\in A^{-}_{j}$. Since for any $l\leq ar(m)$, $a_{l}\in A^{-}_{i_{l}}$, then for any $l\leq ar(m)$, $a_{l}\in A_{j}$. Hence, $F^{\A}_{m}(a_{1},\ldots, a_{ar(m)})\in A_{j}$. 
Moreover, $F_{m}^{\A}(\h_{i_{1,j}}(a_{1}),\ldots,a_{l},\ldots,a_{ar(m)}),\allowbreak\ldots,F_{m}^{\A}(\h_{i_{1},j}(a_{1}),\ldots,\h_{i_{l},j}(a_{l}),\allowbreak\ldots,\allowbreak a_{ar(m)}),\ldots,F_{m}^{\A}(\h_{i_{1},j}(a_{1}),\allowbreak\ldots,\allowbreak\h_{i_{l},j}(a_{l}),\allowbreak$ $\ldots,\allowbreak\h_{i_{ar(m)},j}(a_{ar(m)}))\in A_{j}$. 

Assume indirectly that there is  $k<j$ s.t. $F^{\A}_{m}(a_{1},\ldots, a_{ar(m)})\in A_{k}$. By Definition \ref{def:algebra-splitujaca}, $a_{1},\ldots, a_{ar(m)}\in A_{k}$. By Fact \ref{fa:DS->Spl}.5, for any  $l\leq ar(m)$, $i_{l}\leqslant k$. So, $j\leqslant k$. Similarly we can show that there is no $k<j$ s.t. at least one of the following holds: $F_{m}^{\A}(\h_{i_{1},j}(a_{1}),\ldots,a_{l},\ldots,a_{ar(m)})\in A_{k},\ldots,F_{m}^{\A}(\h_{i_{1},j}(a_{1}),\ldots,\allowbreak\h_{i_{l},j}(a_{l}),\allowbreak\ldots,\allowbreak a_{ar(m)})\in A_{k},\ldots,F_{m}^{\A}(\h_{i_{1},j}(a_{1}),\ldots,\h_{i_{l},j}(a_{l}),\ldots,\allowbreak\h_{i_{ar(m)},j}(\allowbreak a_{ar(m)}\allowbreak))\allowbreak\in A_{k}$.  

Thus, by \eqref{def:-A}, $F^{\A}_{m}(a_{1},\ldots, a_{ar(m)})\in A_{j}^{-}\subseteq\bigcup_{j\leqslant k}A_{k}^{-}$. And also: $F_{m}^{\A}(\h_{i_{1},j}(\allowbreak a_{1}),\allowbreak\ldots,\allowbreak a_{l},\ldots,\allowbreak a_{ar(m)})\in A_{j}^{-},\allowbreak\ldots,\allowbreak F_{m}^{\A}(\h_{i_{1}},j(a_{1}),\allowbreak\ldots,\allowbreak\h_{i_{l},j}(a_{l}),\allowbreak\ldots,\allowbreak a_{ar(m)})\allowbreak\!\in\!\allowbreak A_{j}^{-},\allowbreak\ldots,\allowbreak F_{m}^{\A}(\h_{i_{1}},j(a_{1}),\ldots,\h_{i_{l},j}(a_{l}),\ldots,\allowbreak\h_{i_{ar(m)},j}(\allowbreak a_{ar(m)}\allowbreak))\allowbreak\in A_{j}^{-}$.

By the definition of $\leqslant$ and by \eqref{def:-A}, for all $o,p\leq ar(m)$, if $o\neq p$, then $A_{i_{o}}^{-}\neq A_{i_{p}}^{-}$

Hence, by Definition \ref{def:P-homomorfizm}, the definition of $f_{\h}$ and Fact \ref{fa:warunki}, we obtain:
{\small
\begin{flalign*}
F_{m}^{\A}(a_{1},\ldots,a_{l},\ldots,a_{ar(m)})=&
F_{m}^{\A}(f_{\h_{i_{1},j}}(a_{1}),\ldots,f_{\h_{i_{1},j}}(a_{l}),\ldots,f_{\h_{i_{1},j}}(a_{ar(m)}))&\\
=\;& F_{m}^{\A}(\h_{i_{1},j}(a_{1}),\ldots,a_{l},\ldots,a_{ar(m)})&\\
&\;\;\;\;\;\;\;\;\;\; \vdots&\\
=\;& F_{m}^{\A}(f_{\h_{i_{l},j}}(\h_{i_{1},j}(a_{1})),\ldots,f_{\h_{i_{l},j}}(a_{l}),\ldots,f_{\h_{i_{l},j}}(a_{ar(m)}))&\\
=\;& F_{m}^{\A}(\h_{i_{1},j}(a_{1}),\ldots,\h_{i_{l},j}(a_{l}),\ldots,a_{ar(m)})&\\
&\;\;\;\;\;\;\;\;\;\; \vdots&\\
=&\; F_{m}^{\A}(\h_{i_{1},j}(a_{1}),\ldots,\h_{i_{l},j}(a_{l}),\ldots,\h_{i_{ar(m)},j}(a_{ar(m)})&\\
=&\; F_{m}^{\A_{j}}(\h_{i_{1},j}(a_{1}),\ldots,\h_{i_{l},j}(a_{l}),\ldots,\h_{i_{ar(m)},j}(a_{ar(m)}).\hspace{0.6cm}\qedhere&
\end{flalign*}
}
\end{proof}

By Theorems \ref{pierwsze twierdzenie o dekompozycji} and \ref{Twierdzenie o dekompozycji} the following conclusion follows:

\begin{wn}\label{wn:decomposition}
Algebra $\A=\langle A; F_{1}^{\A},\ldots,F_{n}^{\A}\rangle$ is a Płonka sum iff there is a non-trivial system of algebras isolated wrt $\A$.
\end{wn}

By means of Corollaries \ref{wn:sum-Plonki-wtw-p-function} and \ref{wn:decomposition} we receive the following fact:

\begin{wn}\label{wn:decomposition}
There is a non-trivial system of algebras isolated wrt $\A$ iff there is a partition function of $\A$.
\end{wn}

\section{Two Examples of Decomposition}
\label{Section 5}

In this section, we will apply our decomposition method to two algebras. We will show how, we are able to determine all (up to isomorphism) direct systems of a given Płonka sum as well as proving that a given algebra is not a Płonka sum if it is impossible to determine a system of algebras isolated wrt this algebra.

Let us consider two algebras:
\begin{itemize}

\item $\A_{7}=\langle\{a_{1},\ldots,a_{11}\};\ast\rangle$, 


\item $\A_{8}=\langle\{a_{1},\ldots,a_{10}\};\ast\rangle$,

\end{itemize}

\noindent with binary operations defined by Tables \ref{Def.oper.A7} and \ref{Def.oper.A8}, respectively.

\begin{table}[h!]
\centering
\special{em:linewidth 0.4pt}
\linethickness{0.5pt}
\begin{tabular}{||c||c|c|c|c|c|c|c|c|c|c|c||}
\hhline{|t:=:t:===========:t|}
$\ast$ & $a_{1}$ & $a_{2}$ & $a_{3}$ & $a_{4}$ & $a_{5}$ & $a_{6}$ & $a_{7}$ & $a_{8}$ & $a_{9}$ & $a_{10}$ & $a_{11}$\\
\hhline{|:=::===========:|}
$a_{1}$      & $a_{2}$ & $a_{1}$ & $a_{6}$ & $a_{5}$ & $a_{6}$ & $a_{5}$ & $a_{8}$ & $a_{7}$ & $a_{9}$ & $a_{11}$ & $a_{10}$\\
\hhline{||-||-|-|-|-|-|-|-|-|-|-|-||}
$a_{2}$      & $a_{2}$ & $a_{1}$ & $a_{6}$ & $a_{5}$ & $a_{6}$ & $a_{5}$ & $a_{8}$ & $a_{7}$ & $a_{9}$ & $a_{11}$ & $a_{10}$\\
\hhline{||-||-|-|-|-|-|-|-|-|-|-|-||}
$a_{3}$      & $a_{6}$ & $a_{5}$ & $a_{4}$ & $a_{3}$ & $a_{6}$ & $a_{5}$ & $a_{8}$ & $a_{7}$ & $a_{9}$ & $a_{11}$ & $a_{10}$\\
\hhline{||-||-|-|-|-|-|-|-|-|-|-|-||}
$a_{4}$      & $a_{6}$ & $a_{5}$ & $a_{4}$ & $a_{3}$ & $a_{6}$ & $a_{5}$ & $a_{8}$ & $a_{7}$ & $a_{9}$ & $a_{11}$ & $a_{10}$\\
\hhline{||-||-|-|-|-|-|-|-|-|-|-|-||}
$a_{5}$      & $a_{6}$ & $a_{5}$ & $a_{6}$ & $a_{5}$ & $a_{6}$ & $a_{5}$ & $a_{8}$ & $a_{7}$ & $a_{9}$ & $a_{11}$ & $a_{10}$\\
\hhline{||-||-|-|-|-|-|-|-|-|-|-|-||}
$a_{6}$      & $a_{6}$ & $a_{5}$ & $a_{6}$ & $a_{5}$ & $a_{6}$ & $a_{5}$ & $a_{8}$ & $a_{7}$ & $a_{9}$ & $a_{11}$ & $a_{10}$\\
\hhline{||-||-|-|-|-|-|-|-|-|-|-|-||}
$a_{7}$      & $a_{8}$ & $a_{7}$ & $a_{8}$ & $a_{7}$ & $a_{8}$ & $a_{7}$ & $a_{8}$ & $a_{7}$ & $a_{10}$ & $a_{11}$ & $a_{10}$\\
\hhline{||-||-|-|-|-|-|-|-|-|-|-|-||}
$a_{8}$      & $a_{8}$ & $a_{7}$ & $a_{8}$ & $a_{7}$ & $a_{8}$ & $a_{7}$ & $a_{8}$ & $a_{7}$ & $a_{10}$ & $a_{11}$ & $a_{10}$\\
\hhline{||-||-|-|-|-|-|-|-|-|-|-|-||}
$a_{9}$      & $a_{9}$ & $a_{9}$ & $a_{9}$ & $a_{9}$ & $a_{9}$ & $a_{9}$ & $a_{10}$ & $a_{11}$ & $a_{9}$ & $a_{11}$ & $a_{10}$\\
\hhline{||-||-|-|-|-|-|-|-|-|-|-|-||}
$a_{10}$     & $a_{10}$ & $a_{11}$ & $a_{10}$ & $a_{11}$ & $a_{10}$ & $a_{11}$ & $a_{10}$ & $a_{11}$ & $a_{10}$ & $a_{11}$ & $a_{10}$ \\
\hhline{||-||-|-|-|-|-|-|-|-|-|-|-||}
$a_{11}$     & $a_{10}$ & $a_{11}$ & $a_{10}$ & $a_{11}$ & $a_{10}$ & $a_{11}$ & $a_{10}$ & $a_{11}$ & $a_{10}$ & $a_{11}$ & $a_{10}$ \\
\hhline{|b:=:b:===========:b|}
\end{tabular}
\caption{\label{Def.oper.A7}Operation $\ast$ of algebra $\A_{7}$}
\end{table}


\begin{table}[h!]
\centering
\special{em:linewidth 0.4pt}
\linethickness{0.5pt}
\begin{tabular}{||c||c|c|c|c|c|c|c|c|c|c||}
\hhline{|t:=:t:==========:t|}
$\ast$ & $a_{1}$ & $a_{2}$ & $a_{3}$ & $a_{4}$ & $a_{5}$ & $a_{6}$ & $a_{7}$ & $a_{8}$ & $a_{9}$ & $a_{10}$\\
\hhline{|:=::==========:|}
$a_{1}$      & $a_{1}$ & $a_{3}$ & $a_{3}$ & $a_{3}$ & $a_{6}$ & $a_{6}$ & $a_{8}$ & $a_{8}$ & $a_{9}$ & $a_{9}$ \\
\hhline{||-||-|-|-|-|-|-|-|-|-|-||}
$a_{2}$      & $a_{3}$ & $a_{2}$ & $a_{3}$ & $a_{3}$ & $a_{6}$ & $a_{6}$ & $a_{8}$ & $a_{8}$ & $a_{9}$ & $a_{9}$ \\
\hhline{||-||-|-|-|-|-|-|-|-|-|-||}
$a_{3}$      & $a_{3}$ & $a_{3}$ & $a_{4}$ & $a_{3}$ & $a_{5}$ & $a_{5}$ & $a_{7}$ & $a_{7}$ & $a_{10}$ & $a_{10}$ \\
\hhline{||-||-|-|-|-|-|-|-|-|-|-||}
$a_{4}$      & $a_{3}$ & $a_{3}$ & $a_{3}$ & $a_{3}$ & $a_{5}$ & $a_{5}$ & $a_{7}$ & $a_{7}$ & $a_{10}$ & $a_{10}$ \\
\hhline{||-||-|-|-|-|-|-|-|-|-|-||}
$a_{5}$      & $a_{6}$ & $a_{6}$ & $a_{5}$ & $a_{5}$ & $a_{5}$ & $a_{5}$ & $a_{9}$ & $a_{9}$ & $a_{9}$ & $a_{9}$ \\
\hhline{||-||-|-|-|-|-|-|-|-|-|-||}
$a_{6}$      & $a_{6}$ & $a_{6}$ & $a_{5}$ & $a_{5}$ & $a_{5}$ & $a_{5}$ & $a_{9}$ & $a_{9}$ & $a_{9}$ & $a_{9}$ \\
\hhline{||-||-|-|-|-|-|-|-|-|-|-||}
$a_{7}$      & $a_{8}$ & $a_{8}$ & $a_{7}$ & $a_{7}$ & $a_{9}$ & $a_{9}$ & $a_{7}$ & $a_{7}$ & $a_{9}$ & $a_{9}$ \\
\hhline{||-||-|-|-|-|-|-|-|-|-|-||}
$a_{8}$      & $a_{8}$ & $a_{8}$ & $a_{7}$ & $a_{7}$ & $a_{9}$ & $a_{9}$ & $a_{7}$ & $a_{7}$ & $a_{9}$ & $a_{9}$ \\
\hhline{||-||-|-|-|-|-|-|-|-|-|-||}
$a_{9}$      & $a_{9}$ & $a_{9}$ & $a_{9}$ & $a_{9}$ & $a_{9}$ & $a_{9}$ & $a_{9}$ & $a_{9}$ & $a_{10}$ & $a_{9}$ \\
\hhline{||-||-|-|-|-|-|-|-|-|-|-||}
$a_{10}$     & $a_{9}$ & $a_{9}$ & $a_{9}$ & $a_{9}$ & $a_{9}$ & $a_{9}$ & $a_{9}$ & $a_{9}$ & $a_{9}$ & $a_{9}$ \\
\hhline{|b:=:b:==========:b|}
\end{tabular}
\caption{\label{Def.oper.A8}Operation $\ast$ of algebra $\A_{8}$}
\end{table}

Our decomposition can be performed in two steps. The first step consists of three components. First, we determine the family of algebras isolated wrt a given algebra. Secondly, we select all frames of isolated algebras. Thirdly, we modify the obtained frames by replacing the algebras of the given frame with complement algebras. In the second step, we try to verify if a family of P-homomorphisms can be selected.

\textbf{Step 1.1.} Let us determine all algebras isolated wrt algebras we are considering.

\begin{enumerate}

\item In the case of algebra $\A_{7}$ we have the following isolated algebras:

\begin{itemize}

\item $\A_{7}^{1}$ with the universe $\{a_{1},a_{2}\}$,

\item $\A_{7}^{2}$ with the universe $\{a_{3},a_{4}\}$,

\item $\A_{7}^{3}$ with the universe $\{a_{1},a_{2},a_{3},a_{4},a_{5},a_{6}\}$,

\item $\A_{7}^{4}$ with the universe $\{a_{1},a_{2},a_{3},a_{4},a_{5},a_{6},a_{7},a_{8}\}$,

\item $\A_{7}^{5}$ with the universe $\{a_{1},a_{2},a_{3},a_{4},a_{5},a_{6},a_{9}\}$,

\item $\A_{7}^{6}$ with the universe $\{a_{1},\ldots,a_{11}\}$.

\end{itemize}

\item In the case of algebra $\A_{8}$ we have the following isolated algebras:

\begin{itemize}

\item $\A_{8}^{1}$ with the universe $\{a_{1}\}$,

\item $\A_{8}^{2}$ with the universe $\{a_{2}\}$,

\item $\A_{8}^{3}$ with the universe $\{a_{1}, a_{2}, a_{3},a_{4}\}$,

\item $\A_{8}^{4}$ with the universe $\{a_{1}, a_{2}, a_{3},a_{4},a_{5},a_{6}\}$,

\item $\A_{8}^{5}$ with the universe $\{a_{1}, a_{2}, a_{3},a_{4},a_{7},a_{8}\}$,

\item $\A_{8}^{6}$ with the universe $\{a_{1},\ldots,a_{10}\}$.

\end{itemize} 
 
\end{enumerate} 

Let us adopt the following notation:
${\A_{n}^{m}}^{-X}:=\langle {A_{n}^{m}}^{-X},\ast\rangle$,
where $X\subseteq\{i\in\mathbb{N}:i<m\}$ and ${A^{m}_{n}}^{-X}=A^{m}_{n}\setminus \bigcup_{i\in X}A^{i}_{n}$.

\textbf{Step 1.2.} In Figure \ref{Spl-frames}, we present all frames of algebras isolated wrt $\A_{7}$ and $\A_{8}$, respectively.

\begin{figure}[htbp]
\centering



\begin{tikzpicture}[scale=0.51]

\filldraw[black] (0,0) circle (2pt);
\filldraw[black] (0,0) circle (0pt) node[below]{\footnotesize $\A^{1}_{n}$};
\filldraw[black] (4,0) circle (2pt);	
\filldraw[black] (4,0) circle (0pt) node[below]{\footnotesize $\A^{2}_{n}$};

\draw[black, thick] (0,0) -- (2,2);
\draw[black, thick] (4,0) -- (2,2);

\filldraw[black] (2,2) circle (2pt);
\filldraw[black] (1.8,2) circle (0pt) node[left]{\footnotesize $\A^{3}_{n}$};

\draw[black, thick] (2,2) -- (4,4);
\draw[black, thick] (2,2) -- (0,4);

\filldraw[black] (0,4) circle (2pt);
\filldraw[black] (0,4) circle (0pt) node[left]{\footnotesize $\A^{4}_{n}$};
\filldraw[black] (4,4) circle (2pt);
\filldraw[black] (4,4) circle (0pt) node[right]{\footnotesize $\A^{5}_{n}$};

\draw[black, thick] (0,4) -- (2,6);
\draw[black, thick] (4,4) -- (2,6);

\filldraw[black] (2,6) circle (2pt);
\filldraw[black] (2,6) circle (0pt) node[above]{\footnotesize $\A^{6}_{n}$};

\filldraw[black] (0,6) circle (0pt) node{(1)};


\filldraw[black] (6,0) circle (2pt);
\filldraw[black] (5.7,0) circle (0pt) node[above]{\footnotesize $\A^{1}_{n}$};
\filldraw[black] (10,0) circle (2pt);
\filldraw[black] (10.2,0) circle (0pt) node[above]{\footnotesize $\A^{2}_{n}$};

\draw[black, thick] (6,0) -- (8,2);
\draw[black, thick] (10,0) -- (8,2);

\filldraw[black] (8,2) circle (2pt);
\filldraw[black] (7.8,2) circle (0pt) node[left]{\footnotesize $\A^{3}_{n}$};

\draw[black, thick] (8,2) -- (8,4);

\filldraw[black] (8,4) circle (2pt);
\filldraw[black] (8,4) circle (0pt) node[left]{\footnotesize $\A^{5}_{n}$};

\draw[black, thick] (8,4) -- (8,6);

\filldraw[black] (8,6) circle (2pt);
\filldraw[black] (8,6) circle (0pt) node[above]{\footnotesize $\A^{6}_{n}$};

\filldraw[black] (6,6) circle (0pt) node{(2)};


\filldraw[black] (12,0) circle (2pt);
\filldraw[black] (11.7,0) circle (0pt) node[above]{\footnotesize $\A^{1}_{n}$};
\filldraw[black] (16,0) circle (2pt);
\filldraw[black] (16.2,0) circle (0pt) node[above]{\footnotesize $\A^{2}_{n}$};

\draw[black, thick] (12,0) -- (14,2);
\draw[black, thick] (16,0) -- (14,2);

\filldraw[black] (14,2) circle (2pt);
\filldraw[black] (13.8,2) circle (0pt) node[left]{\footnotesize $\A^{3}_{n}$};

\draw[black, thick] (14,2) -- (14,4);

\filldraw[black] (14,4) circle (2pt);
\filldraw[black] (14,4) circle (0pt) node[left]{\footnotesize $\A^{4}_{n}$};

\draw[black, thick] (14,4) -- (14,6);

\filldraw[black] (14,6) circle (2pt);
\filldraw[black] (14,6) circle (0pt) node[above]{\footnotesize $\A^{6}_{n}$};

\filldraw[black] (12,6) circle (0pt) node{(3)};

\filldraw[black] (20,0) circle (2pt);
\filldraw[black] (19.8,0) circle (0pt) node[left]{\footnotesize $\A^{2}_{n}$};

\draw[black, thick] (20,0) -- (20,2);

\filldraw[black] (20,2) circle (2pt);
\filldraw[black] (19.8,1.9) circle (0pt) node[left]{\footnotesize $\A^{3}_{n}$};

\draw[black, thick] (20,2) -- (18,4);
\draw[black, thick] (20,2) -- (22,4);

\filldraw[black] (18,4) circle (2pt);
\filldraw[black] (18,4) circle (0pt) node[left]{\footnotesize $\A^{4}_{n}$};
\filldraw[black] (22,4) circle (2pt);
\filldraw[black] (22,4) circle (0pt) node[right]{\footnotesize $\A^{5}_{n}$};

\draw[black, thick] (18,4) -- (20,6);
\draw[black, thick] (22,4) -- (20,6);

\filldraw[black] (20,6) circle (2pt);
\filldraw[black] (20,6) circle (0pt) node[above]{\footnotesize $\A^{6}_{n}$};

\filldraw[black] (18,6) circle (0pt) node{(4)};


\filldraw[black] (2,-8) circle (2pt);
\filldraw[black] (2,-8) circle (0pt) node[right]{\footnotesize $\A^{1}_{n}$};

\draw[black, thick] (2,-8) -- (2,-6);

\filldraw[black] (2,-6) circle (2pt);
\filldraw[black] (1.8,-6.1) circle (0pt) node[left]{\footnotesize $\A^{3}_{n}$};

\draw[black, thick] (2,-6) -- (0,-4);
\draw[black, thick] (2,-6) -- (4,-4);

\filldraw[black] (0,-4) circle (2pt);
\filldraw[black] (0,-4) circle (0pt) node[left]{\footnotesize $\A^{4}_{n}$};
\filldraw[black] (4,-4) circle (2pt);
\filldraw[black] (3.9,-4) circle (0pt) node[right]{\footnotesize $\A^{5}_{n}$};

\draw[black, thick] (0,-4) -- (2,-2);
\draw[black, thick] (4,-4) -- (2,-2);

\filldraw[black] (2,-2) circle (2pt);
\filldraw[black] (2,-2) circle (0pt) node[above]{\footnotesize $\A^{6}_{n}$};

\filldraw[black] (0,-2) circle (0pt) node{(5)};

\filldraw[black] (8,-8) circle (2pt);
\filldraw[black] (8,-8) circle (0pt) node[left]{\footnotesize $\A^{2}_{n}$};

\draw[black, thick] (8,-8) -- (8,-6);

\filldraw[black] (8,-6) circle (2pt);
\filldraw[black] (8,-6) circle (0pt) node[left]{\footnotesize $\A^{3}_{n}$};

\draw[black, thick] (8,-6) -- (8,-4);

\filldraw[black] (8,-4) circle (2pt);
\filldraw[black] (8,-4) circle (0pt) node[left]{\footnotesize $\A^{5}_{n}$};

\draw[black, thick] (8,-4) -- (8,-2);

\filldraw[black] (8,-2) circle (2pt);
\filldraw[black] (8,-2) circle (0pt) node[above]{\footnotesize $\A^{6}_{n}$};

\filldraw[black] (6,-2) circle (0pt) node{(6)};

\filldraw[black] (12,-8) circle (2pt);
\filldraw[black] (12,-8) circle (0pt) node[left]{\footnotesize $\A^{1}_{n}$};

\draw[black, thick] (12,-8) -- (12,-6);

\filldraw[black] (12,-6) circle (2pt);
\filldraw[black] (12,-6) circle (0pt) node[left]{\footnotesize $\A^{3}_{n}$};

\draw[black, thick] (12,-6) -- (12,-4);

\filldraw[black] (12,-4) circle (2pt);
\filldraw[black] (12,-4) circle (0pt) node[left]{\footnotesize $\A^{5}_{n}$};

\draw[black, thick] (12,-4) -- (12,-2);

\filldraw[black] (12,-2) circle (2pt);
\filldraw[black] (12,-2) circle (0pt) node[above]{\footnotesize $\A^{6}_{n}$};

\filldraw[black] (10,-2) circle (0pt) node{(7)};

\filldraw[black] (16,-8) circle (2pt);
\filldraw[black] (16,-8) circle (0pt) node[left]{\footnotesize $\A^{2}_{n}$};

\draw[black, thick] (16,-8) -- (16,-6);

\filldraw[black] (16,-6) circle (2pt);
\filldraw[black] (16,-6) circle (0pt) node[left]{\footnotesize $\A^{3}_{n}$};

\draw[black, thick] (16,-6) -- (16,-4);

\filldraw[black] (16,-4) circle (2pt);
\filldraw[black] (16,-4) circle (0pt) node[left]{\footnotesize $\A^{4}_{n}$};

\draw[black, thick] (16,-4) -- (16,-2);

\filldraw[black] (16,-2) circle (2pt);
\filldraw[black] (16,-2) circle (0pt) node[above]{\footnotesize $\A^{6}_{n}$};

\filldraw[black] (14,-2) circle (0pt) node{(8)};

\filldraw[black] (20,-8) circle (2pt);
\filldraw[black] (20,-8) circle (0pt) node[left]{\footnotesize $\A^{1}_{n}$};

\draw[black, thick] (20,-8) -- (20,-6);

\filldraw[black] (20,-6) circle (2pt);
\filldraw[black] (20,-6) circle (0pt) node[left]{\footnotesize $\A^{3}_{n}$};

\draw[black, thick] (20,-6) -- (20,-4);

\filldraw[black] (20,-4) circle (2pt);
\filldraw[black] (20,-4) circle (0pt) node[left]{\footnotesize $\A^{4}_{n}$};

\draw[black, thick] (20,-4) -- (20,-2);

\filldraw[black] (20,-2) circle (2pt);
\filldraw[black] (20,-2) circle (0pt) node[above]{\footnotesize $\A^{6}_{n}$};

\filldraw[black] (18,-2) circle (0pt) node{(9)};


\filldraw[black] (0,-14) circle (2pt);
\filldraw[black] (0,-14) circle (0pt) node[below]{\footnotesize $\A^{1}_{n}$};
\filldraw[black] (4,-14) circle (2pt);
\filldraw[black] (4,-14) circle (0pt) node[below]{\footnotesize $\A^{2}_{n}$};

\draw[black, thick] (0,-14) -- (2,-12);
\draw[black, thick] (4,-14) -- (2,-12);

\filldraw[black] (2,-12) circle (2pt);
\filldraw[black] (2,-11.9) circle (0pt) node[left]{\footnotesize $\A^{3}_{n}$};

\draw[black, thick] (2,-12) -- (2,-10);

\filldraw[black] (2,-10) circle (2pt);
\filldraw[black] (2,-10) circle (0pt) node[above]{\footnotesize $\A^{6}_{n}$};

\filldraw[black] (0,-10) circle (0pt) node{(10)};

\filldraw[black] (6,-14) circle (2pt);
\filldraw[black] (5.8,-14) circle (0pt) node[above]{\footnotesize $\A^{1}_{n}$};
\filldraw[black] (10,-14) circle (2pt);
\filldraw[black] (10.2,-14) circle (0pt) node[above]{\footnotesize $\A^{2}_{n}$};

\draw[black, thick] (6,-14) -- (8,-12);
\draw[black, thick] (10,-14) -- (8,-12);

\filldraw[black] (8,-12) circle (2pt);
\filldraw[black] (8,-11.9) circle (0pt) node[left]{\footnotesize $\A^{5}_{n}$};

\draw[black, thick] (8,-12) -- (8,-10);

\filldraw[black] (8,-10) circle (2pt);
\filldraw[black] (8,-10) circle (0pt) node[above]{\footnotesize $\A^{6}_{n}$};

\filldraw[black] (6,-10) circle (0pt) node{(11)};

\filldraw[black] (12,-14) circle (2pt);
\filldraw[black] (12,-14) circle (0pt) node[below]{\footnotesize $\A^{1}_{n}$};
\filldraw[black] (16,-14) circle (2pt);
\filldraw[black] (16,-14) circle (0pt) node[below]{\footnotesize $\A^{2}_{n}$};

\draw[black, thick] (12,-14) -- (14,-12);
\draw[black, thick] (16,-14) -- (14,-12);

\filldraw[black] (14,-12) circle (2pt);
\filldraw[black] (14,-11.9) circle (0pt) node[left]{\footnotesize $\A^{4}_{n}$};

\draw[black, thick] (14,-12) -- (14,-10);

\filldraw[black] (14,-10) circle (2pt);
\filldraw[black] (14,-10) circle (0pt) node[above]{\footnotesize $\A^{6}_{n}$};

\filldraw[black] (12,-10) circle (0pt) node{(12)};

\filldraw[black] (20,-14) circle (2pt);
\filldraw[black] (20,-14) circle (0pt) node[left]{\footnotesize $\A^{2}_{n}$};

\draw[black, thick] (20,-14) -- (20,-12);

\filldraw[black] (20,-12) circle (2pt);
\filldraw[black] (20,-12) circle (0pt) node[left]{\footnotesize $\A^{3}_{n}$};

\draw[black, thick] (20,-12) -- (20,-10);

\filldraw[black] (20,-10) circle (2pt);
\filldraw[black] (20,-10) circle (0pt) node[above]{\footnotesize $\A^{6}_{n}$};

\filldraw[black] (18,-10) circle (0pt) node{(13)};


\filldraw[black] (2,-20) circle (2pt);
\filldraw[black] (2,-20) circle (0pt) node[left]{\footnotesize $\A^{1}_{n}$};

\draw[black, thick] (2,-20) -- (2,-18);

\filldraw[black] (2,-18) circle (2pt);
\filldraw[black] (2,-18) circle (0pt) node[left]{\footnotesize $\A^{3}_{n}$};

\draw[black, thick] (2,-18) -- (2,-16);

\filldraw[black] (2,-16) circle (2pt);
\filldraw[black] (2,-16) circle (0pt) node[above]{\footnotesize $\A^{6}_{n}$};

\filldraw[black] (0,-16) circle (0pt) node{(14)};

\filldraw[black] (6,-20) circle (2pt);
\filldraw[black] (6,-20) circle (0pt) node[left]{\footnotesize $\A^{2}_{n}$};

\draw[black, thick] (6,-20) -- (6,-18);

\filldraw[black] (6,-18) circle (2pt);
\filldraw[black] (6,-18) circle (0pt) node[left]{\footnotesize $\A^{5}_{n}$};

\draw[black, thick] (6,-18) -- (6,-16);

\filldraw[black] (6,-16) circle (2pt);
\filldraw[black] (6,-16) circle (0pt) node[above]{\footnotesize $\A^{6}_{n}$};

\filldraw[black] (4,-16) circle (0pt) node{(15)};

\filldraw[black] (10,-20) circle (2pt);
\filldraw[black] (10,-20) circle (0pt) node[left]{\footnotesize $\A^{1}_{n}$};

\draw[black, thick] (10,-20) -- (10,-18);

\filldraw[black] (10,-18) circle (2pt);
\filldraw[black] (10,-18) circle (0pt) node[left]{\footnotesize $\A^{5}_{n}$};

\draw[black, thick] (10,-18) -- (10,-16);

\filldraw[black] (10,-16) circle (2pt);
\filldraw[black] (10,-16) circle (0pt) node[above]{\footnotesize $\A^{6}_{n}$};

\filldraw[black] (8,-16) circle (0pt) node{(16)};

\filldraw[black] (14,-20) circle (2pt);
\filldraw[black] (14,-20) circle (0pt) node[left]{\footnotesize $\A^{3}_{n}$};

\draw[black, thick] (14,-20) -- (14,-18);

\filldraw[black] (14,-18) circle (2pt);
\filldraw[black] (14,-18) circle (0pt) node[left]{\footnotesize $\A^{5}_{n}$};

\draw[black, thick] (14,-18) -- (14,-16);

\filldraw[black] (14,-16) circle (2pt);
\filldraw[black] (14,-16) circle (0pt) node[above]{\footnotesize $\A^{6}_{n}$};

\filldraw[black] (12,-16) circle (0pt) node{(17)};

\filldraw[black] (18,-20) circle (2pt);
\filldraw[black] (18,-20) circle (0pt) node[left]{\footnotesize $\A^{2}_{n}$};

\draw[black, thick] (18,-20) -- (18,-18);

\filldraw[black] (18,-18) circle (2pt);
\filldraw[black] (18,-18) circle (0pt) node[left]{\footnotesize $\A^{4}_{n}$};

\draw[black, thick] (18,-18) -- (18,-16);

\filldraw[black] (18,-16) circle (2pt);
\filldraw[black] (18,-16) circle (0pt) node[above]{\footnotesize $\A^{6}_{n}$};

\filldraw[black] (16,-16) circle (0pt) node{(18)};

\filldraw[black] (22,-20) circle (2pt);
\filldraw[black] (22,-20) circle (0pt) node[left]{\footnotesize $\A^{1}_{n}$};

\draw[black, thick] (22,-20) -- (22,-18);

\filldraw[black] (22,-18) circle (2pt);
\filldraw[black] (22,-18) circle (0pt) node[left]{\footnotesize $\A^{4}_{n}$};

\draw[black, thick] (22,-18) -- (22,-16);

\filldraw[black] (22,-16) circle (2pt);
\filldraw[black] (22,-16) circle (0pt) node[above]{\footnotesize $\A^{6}_{n}$};

\filldraw[black] (20,-16) circle (0pt) node{(19)};


\filldraw[black] (2,-26) circle (2pt);
\filldraw[black] (2,-26) circle (0pt) node[left]{\footnotesize $\A^{3}_{n}$};

\draw[black, thick] (2,-26) -- (2,-24);

\filldraw[black] (2,-24) circle (2pt);
\filldraw[black] (2,-24) circle (0pt) node[left]{\footnotesize $\A^{4}_{n}$};

\draw[black, thick] (2,-24) -- (2,-22);

\filldraw[black] (2,-22) circle (2pt);
\filldraw[black] (2,-22) circle (0pt) node[above]{\footnotesize $\A^{6}_{n}$};

\filldraw[black] (0,-22) circle (0pt) node{(20)};

\filldraw[black] (6,-26) circle (2pt);
\filldraw[black] (5.8,-26) circle (0pt) node[left]{\footnotesize $\A^{3}_{n}$};

\draw[black, thick] (6,-26) -- (4,-24);
\draw[black, thick] (6,-26) -- (8,-24);

\filldraw[black] (4,-24) circle (2pt);
\filldraw[black] (4.2,-24) circle (0pt) node[right]{\footnotesize $\A^{4}_{n}$};
\filldraw[black] (8,-24) circle (2pt);
\filldraw[black] (7.8,-24) circle (0pt) node[left]{\footnotesize $\A^{5}_{n}$};

\draw[black, thick] (4,-24) -- (6,-22);
\draw[black, thick] (8,-24) -- (6,-22);

\filldraw[black] (6,-22) circle (2pt);
\filldraw[black] (6,-22) circle (0pt) node[above]{\footnotesize $\A^{6}_{n}$};

\filldraw[black] (4,-22) circle (0pt) node{(21)};

\filldraw[black] (10,-24) circle (2pt);
\filldraw[black] (10,-24) circle (0pt) node[below]{\footnotesize $\A^{1}_{n}$};
\filldraw[black] (14,-24) circle (2pt);
\filldraw[black] (14,-24) circle (0pt) node[below]{\footnotesize $\A^{2}_{n}$};

\draw[black, thick] (10,-24) -- (12,-22);
\draw[black, thick] (14,-24) -- (12,-22);

\filldraw[black] (12,-22) circle (2pt);
\filldraw[black] (12,-22) circle (0pt) node[above]{\footnotesize $\A^{6}_{n}$};

\filldraw[black] (10,-22) circle (0pt) node{(22)};

\filldraw[black] (18,-24) circle (2pt);
\filldraw[black] (18,-24) circle (0pt) node[below]{\footnotesize $\A^{2}_{n}$};

\draw[black, thick] (18,-24) -- (18,-22);

\filldraw[black] (18,-22) circle (2pt);
\filldraw[black] (18,-22) circle (0pt) node[above]{\footnotesize $\A^{6}_{n}$};

\filldraw[black] (16,-22) circle (0pt) node{(23)};

\filldraw[black] (22,-24) circle (2pt);
\filldraw[black] (22,-24) circle (0pt) node[below]{\footnotesize $\A^{1}_{n}$};

\draw[black, thick] (22,-24) -- (22,-22);

\filldraw[black] (22,-22) circle (2pt);
\filldraw[black] (22,-22) circle (0pt) node[above]{\footnotesize $\A^{6}_{n}$};

\filldraw[black] (20,-22) circle (0pt) node{(24)};


\filldraw[black] (2,-30) circle (2pt);
\filldraw[black] (2,-30) circle (0pt) node[below]{\footnotesize $\A^{3}_{n}$};

\draw[black, thick] (2,-30) -- (2,-28);

\filldraw[black] (2,-28) circle (2pt);
5\filldraw[black] (2,-28) circle (0pt) node[above]{\footnotesize $\A^{6}_{n}$};

\filldraw[black] (0,-28) circle (0pt) node{(25)};

\filldraw[black] (6,-30) circle (2pt);
\filldraw[black] (6,-30) circle (0pt) node[below]{\footnotesize $\A^{5}_{n}$};

\draw[black, thick] (6,-30) -- (6,-28);

\filldraw[black] (6,-28) circle (2pt);
\filldraw[black] (6,-28) circle (0pt) node[above]{\footnotesize $\A^{6}_{n}$};

\filldraw[black] (4,-28) circle (0pt) node{(26)};

\filldraw[black] (10,-30) circle (2pt);
\filldraw[black] (10,-30) circle (0pt) node[below]{\footnotesize $\A^{4}_{n}$};

\draw[black, thick] (10,-30) -- (10,-28);

5\filldraw[black] (10,-28) circle (2pt);
\filldraw[black] (10,-28) circle (0pt) node[above]{\footnotesize $\A^{6}_{n}$};

\filldraw[black] (8,-28) circle (0pt) node{(27)};

\end{tikzpicture}

\caption{\label{Spl-frames}Frames of algebras isolated wrt $\A_{7}$ and $\A_{8}$}
\end{figure}

\textbf{Step 1.3.} Considering any frame presented in Figure \ref{Spl-frames}, we replace any algebra $\A_{n}^{m}$ in this frame with algebra ${\A_{n}^{m}}^{-}$, where ${A_{n}^{m}}^{-}={A_{n}^{m}}\setminus\bigcup_{i<m}A_{n}^{i}$.

\textbf{Step 2.} We proceed to the analysis of the possibility of determining P-homomorphisms. 

Let us start with algebra $\A_{7}$. We have the following P-homomorphisms:
\begin{itemize}
\item $\h_{1,3}(a_{1})=a_{5}$ and $\h_{1,3}(a_{2})=a_{6}$, $\h_{1,4}(a_{1})=a_{7}$ and $\h_{1,4}(a_{2})=a_{8}$, 
$\h_{1,6}(a_{1})=a_{11}$ and $\h_{1,6}(a_{2})=a_{10}$;
\item $\h_{2,3}(a_{3})=a_{5}$ and $\h_{2,3}(a_{4})=a_{6}$, $\h_{2,4}(a_{3})=a_{7}$ and $\h_{2,4}(a_{4})=a_{8}$, 
$\h_{2,6}(a_{3})=a_{11}$ and $\h_{2,6}(a_{4})=a_{10}$;
\item $\h_{3,4}(a_{5})=a_{7}$ and $\h_{3,4}(a_{6})=a_{8}$, 
$\h_{3,6}(a_{5})=a_{11}$ and $\h_{3,6}(a_{6})=a_{10}$;
\item $\h_{4,6}(a_{7})=a_{11}$ and $\h_{4,6}(a_{8})=a_{12}$, 
\end{itemize}
where $\h_{i,j}\in\Hom({\A_{7}^{i}}^{-},{\A_{7}^{j}}^{-})$, for the considered indexes $i, j$.

Note that there is no homomorphism from ${\A_{7}^{5}}^{-}$ to ${\A_{7}^{6}}\setminus\A_{7}^{5}$. By Table \ref{Def.oper.A7}, $a_{9}\ast a_{9}=a_{9}$. However, for every $a\in A_{7}^{6}\setminus A_{7}^{5}$, $a\ast a\neq a$. Thus, frames (1), (2), (4), (5), (6), (7), (11), (14), (15), (16), (21) and (26) form Figure \ref{Spl-frames}, are not those over which we can determine systems of algebras isolated wrt $\A_{7}$. In the case of the remaining frames from Figure \ref{Spl-frames}, the sound sets of P-homomorphisms can be determined.

The possible to define systems of algebras isolated wrt $\A_{7}$ we present in Figure \ref{A7-Spl-System}. By Corollary \ref{wn:decomposition} algebra $\A_{7}$ is a Płonka sum.

Let us analyze algebra $\A_{8}$. We will consider case (1) from Figure \ref{Spl-frames} and check whether there are P-homomorphisms from $\A_{8}^{1}$ and $\A_{8}^{2}$ to algebras $\ A_{8}^{n^{-}}$, for any $2<n\leq 6$. 

\begin{itemize}

\item There are no homomorphisms from $\A_{8}^{1^{-}}$ and from $\A_{8}^{2^{-}}$ to $\A_{8}^{3^{-}}$. Similarly, there are no homomorphisms $\A_{8}^{1^{-}}$ and from $\A_{8}^{2^{-}}$ to $\A_{8}^{6^{-}}$. By Table \ref{Def.oper.A8}, $a_{1}\ast a_{1}= a_{1}$ and $a_{2}\ast a_{2}= a_{2}$. In turn, $a_{3}\ast a_{3}=a_{4}$ and $a_{4}\ast a_{4}=a_{3}$. An analogous dependence occurs when trying to determine the homomorphism from $\A_{8}^{1^{-}}$ and from $\A_{8}^{2^{-}}$ to $\A_{8}^{6^{-}}$.

\item There is one homomorphism from $\A_{8}^{1^{-}}$ and one from $\A_{8}^{2^{-}}$ to $\A_{8}^{4^{-}}$. Similarly, there is one homomorphism from $\A_{8}^{1^{-}}$ and from $\A_{8}^{2^{-}}$ to $\A_{8}^{5^{-}}$. By Table \ref{Def.oper.A8}, $a_{5}\ast a_{5}=a_{5}$. However, such homomorphisms are not P-homomorphisms. By Table \ref{Def.oper.A8}, $a_{1}\ast a_{5}=a_{6}\neq a_{5}\ast a_{5}$ and $a_{2}\ast a_{5}=a_{6}\neq a_{5}\ast a_{5}$. 
An analogous dependence occurs when considering homomorphisms from $\A_{8}^{1^{-}}$ and from $\A_{8}^{2^{-}}$ to $\A_{8}^{6^{-}}$.
\end{itemize} 

According to our analysis, by Corollary \ref{wn:decomposition}, algebra $\A_{8}$ is not a Płonka sum.

\begin{figure}[htbp]
\centering

\begin{tikzpicture}[scale=0.54]

\filldraw[black] (0,0) circle (2pt);
\filldraw[black] (0,0) circle (0pt) node[below]{\footnotesize $\A^{1}_{7}$};
\filldraw[black] (4,0) circle (2pt);
\filldraw[black] (4,0) circle (0pt) node[below]{\footnotesize $\A^{2}_{7}$};

\draw[black, thick, ->] (0,0) -- (1.8,1.8);
\filldraw[black] (0.8,0.8) circle (0pt) node[left]{\footnotesize $\h_{1,3}$};
\draw[black, thick, ->] (4,0) -- (2.2,1.8);
\filldraw[black] (5,0.8) circle (0pt) node[left]{\footnotesize $\h_{2,3}$};

\filldraw[black] (2,2) circle (2pt);
\filldraw[black] (1.8,2) circle (0pt) node[left]{\footnotesize ${\A^{3}_{7}}^{-\{1,2\}}$};

\draw[black, thick, ->] (2,2) -- (2,3.8);
\filldraw[black] (2,3) circle (0pt) node[right]{\footnotesize $\h_{3,4}$};

\filldraw[black] (2,4) circle (2pt);
\filldraw[black] (2,4) circle (0pt) node[left]{\footnotesize ${\A_{7}^{4}}^{-\{3\}}$};

\draw[black, thick, ->] (2,4) -- (2,5.8);
\filldraw[black] (2,5) circle (0pt) node[right]{\footnotesize $\h_{4,6}$};

\filldraw[black] (2,6) circle (2pt);
\filldraw[black] (2,6) circle (0pt) node[above]{\footnotesize ${\A_{7}^{6}}^{-\{4\}}$};

\filldraw[black] (0,6) circle (0pt) node{(1)};

\filldraw[black] (8,0) circle (2pt);
\filldraw[black] (8,0) circle (0pt) node[left]{\footnotesize $\A_{7}^{2}$};

\draw[black, thick,->] (8,0) -- (8,1.8);
\filldraw[black] (8,1) circle (0pt) node[right]{\footnotesize $\h_{2,3}$};

\filldraw[black] (8,2) circle (2pt);
\filldraw[black] (8,2) circle (0pt) node[left]{\footnotesize ${\A_{7}^{3}}^{-\{2\}}$};

\draw[black, thick, ->] (8,2) -- (8,3.8);
\filldraw[black] (8,3) circle (0pt) node[right]{\footnotesize $\h_{3,4}$};

\filldraw[black] (8,4) circle (2pt);
\filldraw[black] (8,4) circle (0pt) node[left]{\footnotesize ${\A_{7}^{4}}^{-\{3\}}$};

\draw[black, thick, ->] (8,4) -- (8,5.8);
\filldraw[black] (8,5) circle (0pt) node[right]{\footnotesize $\h_{4,6}$};

\filldraw[black] (8,6) circle (2pt);
\filldraw[black] (8,6) circle (0pt) node[above]{\footnotesize ${\A_{7}^{6}}^{-\{4\}}$};

\filldraw[black] (6,6) circle (0pt) node{(2)};

\filldraw[black] (13,0) circle (2pt);
\filldraw[black] (13,0) circle (0pt) node[left]{\footnotesize $\A_{7}^{1}$};

\draw[black, thick, ->] (13,0) -- (13,1.8);
\filldraw[black] (13,1) circle (0pt) node[right]{\footnotesize $\h_{1,3}$};

\filldraw[black] (13,2) circle (2pt);
\filldraw[black] (13,2) circle (0pt) node[left]{\footnotesize ${\A_{7}^{3}}^{-\{1\}}$};

\draw[black, thick, ->] (13,2) -- (13,3.8);
\filldraw[black] (13,3) circle (0pt) node[right]{\footnotesize $\h_{3,4}$};

\filldraw[black] (13,4) circle (2pt);
\filldraw[black] (13,4) circle (0pt) node[left]{\footnotesize ${\A_{7}^{4}}^{-\{3\}}$};

\draw[black, thick, ->] (13,4) -- (13,5.8);
\filldraw[black] (13,5) circle (0pt) node[right]{\footnotesize $\h_{4,6}$};

\filldraw[black] (13,6) circle (2pt);
\filldraw[black] (13,6) circle (0pt) node[above]{\footnotesize ${\A_{7}^{6}}^{-\{4\}}$};

\filldraw[black] (11,6) circle (0pt) node{(3)};

\filldraw[black] (16,2) circle (2pt);
\filldraw[black] (16,2) circle (0pt) node[below]{\footnotesize $\A_{7}^{1}$};
\filldraw[black] (20,2) circle (2pt);
\filldraw[black] (20,2) circle (0pt) node[below]{\footnotesize $\A_{7}^{2}$};

\draw[black, thick,->] (16,2) -- (17.8,3.8);
\filldraw[black] (16.6,2.5) circle (0pt) node[right]{\footnotesize $\h_{1,3}$};
\draw[black, thick, ->] (20,2) -- (18.2,3.8);
\filldraw[black] (19.5,2.5) circle (0pt) node[right]{\footnotesize $\h_{2,3}$};

\filldraw[black] (18,4) circle (2pt);
\filldraw[black] (18,4.1) circle (0pt) node[left]{\footnotesize ${\A_{7}^{3}}^{-\{1,2\}}$};

\draw[black, thick] (18,4) -- (18,6);
\filldraw[black] (18,5) circle (0pt) node[right]{\footnotesize $\h_{3,6}$};

\filldraw[black] (18,6) circle (2pt);
\filldraw[black] (18,6) circle (0pt) node[above]{\footnotesize ${\A_{7}^{6}}^{-\{3\}}$};

\filldraw[black] (16,6) circle (0pt) node{(4)};


\filldraw[black] (0,-6) circle (2pt);
\filldraw[black] (0,-6) circle (0pt) node[below]{\footnotesize $\A_{7}^{1}$};
\filldraw[black] (4,-6) circle (2pt);
\filldraw[black] (4,-6) circle (0pt) node[below]{\footnotesize $\A_{7}^{2}$};

\draw[black, thick, ->] (0,-6) -- (1.8,-4.2);
\filldraw[black] (0.8,-5.2) circle (0pt) node[left]{\footnotesize $\h_{1,4}$};
\draw[black, thick, ->] (4,-6) -- (2.2,-4.2);
\filldraw[black] (5,-5.2) circle (0pt) node[left]{\footnotesize $\h_{2,4}$};

\filldraw[black] (2,-4) circle (2pt);
\filldraw[black] (2,-3.9) circle (0pt) node[left]{\footnotesize ${\A_{7}^{4}}^{-\{1,2\}}$};

\draw[black, thick, ->] (2,-4) -- (2,-2.2);
\filldraw[black] (2,-3) circle (0pt) node[right]{\footnotesize $\h_{4,6}$};

\filldraw[black] (2,-2) circle (2pt);
\filldraw[black] (2,-2) circle (0pt) node[above]{\footnotesize ${\A_{7}^{6}}^{-\{4\}}$};

\filldraw[black] (0,-2) circle (0pt) node{(5)};

\filldraw[black] (8,-6) circle (2pt);
\filldraw[black] (8,-6) circle (0pt) node[left]{\footnotesize $\A_{7}^{2}$};

\draw[black, thick, ->] (8,-6) -- (8,-4.2);
\filldraw[black] (8,-5) circle (0pt) node[right]{\footnotesize $\h_{2,3}$};

\filldraw[black] (8,-4) circle (2pt);
\filldraw[black] (8,-4) circle (0pt) node[left]{\footnotesize ${\A_{7}^{3}}^{-\{2\}}$};

\draw[black, thick, ->] (8,-4) -- (8,-2.2);
\filldraw[black] (8,-3) circle (0pt) node[right]{\footnotesize $\h_{3,6}$};

\filldraw[black] (8,-2) circle (2pt);
\filldraw[black] (8,-2) circle (0pt) node[above]{\footnotesize ${\A_{7}^{6}}^{-\{3\}}$};

\filldraw[black] (6,-2) circle (0pt) node{(6)};

\filldraw[black] (13,-6) circle (2pt);
\filldraw[black] (13,-6) circle (0pt) node[left]{\footnotesize $\A_{7}^{1}$};

\draw[black, thick, ->] (13,-6) -- (13,-4.2);
\filldraw[black] (13,-5) circle (0pt) node[right]{\footnotesize $\h_{1,3}$};

\filldraw[black] (13,-4) circle (2pt);
\filldraw[black] (13,-4) circle (0pt) node[left]{\footnotesize ${\A_{7}^{3}}^{-\{1\}}$};

\draw[black, thick, ->] (13,-4) -- (13,-2.2);
\filldraw[black] (13,-3) circle (0pt) node[right]{\footnotesize $\h_{3,6}$};

\filldraw[black] (13,-2) circle (2pt);
\filldraw[black] (13,-2) circle (0pt) node[above]{\footnotesize ${\A_{7}^{6}}^{-\{3\}}$};

\filldraw[black] (11,-2) circle (0pt) node{(7)};

\filldraw[black] (18,-6) circle (2pt);
\filldraw[black] (18,-6) circle (0pt) node[left]{\footnotesize $\A_{7}^{2}$};

\draw[black, thick, ->] (18,-6) -- (18,-4.2);
\filldraw[black] (18,-5) circle (0pt) node[right]{\footnotesize $\h_{2,4}$};

\filldraw[black] (18,-4) circle (2pt);
\filldraw[black] (18,-4) circle (0pt) node[left]{\footnotesize ${\A_{7}^{4}}^{-\{2\}}$};

\draw[black, thick, ->] (18,-4) -- (18,-2.2);
\filldraw[black] (18,-3) circle (0pt) node[right]{\footnotesize $\h_{4,6}$};

\filldraw[black] (18,-2) circle (2pt);
\filldraw[black] (18,-2) circle (0pt) node[above]{\footnotesize ${\A_{7}^{6}}^{-\{4\}}$};

\filldraw[black] (16,-2) circle (0pt) node{(8)};


\filldraw[black] (2,-12) circle (2pt);
\filldraw[black] (2,-12) circle (0pt) node[left]{\footnotesize $\A_{7}^{1}$};

\draw[black, thick, ->] (2,-12) -- (2,-10.2);
\filldraw[black] (2,-11) circle (0pt) node[right]{\footnotesize $\h_{1,4}$};

\filldraw[black] (2,-10) circle (2pt);
\filldraw[black] (2,-10) circle (0pt) node[left]{\footnotesize ${\A_{7}^{4}}^{-\{1\}}$};

\draw[black, thick, ->] (2,-10) -- (2,-8.2);
\filldraw[black] (2,-9) circle (0pt) node[right]{\footnotesize $\h_{4,6}$};

\filldraw[black] (2,-8) circle (2pt);
\filldraw[black] (2,-8) circle (0pt) node[above]{\footnotesize ${\A_{7}^{6}}^{-\{4\}}$};

\filldraw[black] (0,-8) circle (0pt) node{(9)};

\filldraw[black] (8,-12) circle (2pt);
\filldraw[black] (8,-12) circle (0pt) node[left]{\footnotesize $\A_{7}^{3}$};

\draw[black, thick, ->] (8,-12) -- (8,-10.2);
\filldraw[black] (8,-11) circle (0pt) node[right]{\footnotesize $\h_{3,4}$};

\filldraw[black] (8,-10) circle (2pt);
\filldraw[black] (8,-10) circle (0pt) node[left]{\footnotesize ${\A_{7}^{4}}^{-\{3\}}$};

\draw[black, thick, ->] (8,-10) -- (8,-8.2);
\filldraw[black] (8,-9) circle (0pt) node[right]{\footnotesize $\h_{4,6}$};

\filldraw[black] (8,-8) circle (2pt);
\filldraw[black] (8,-8) circle (0pt) node[above]{\footnotesize ${\A_{7}^{6}}^{-\{4\}}$};

\filldraw[black] (6,-8) circle (0pt) node{(10)};

\filldraw[black] (11,-10) circle (2pt);
\filldraw[black] (11,-10) circle (0pt) node[below]{\footnotesize $\A_{7}^{1}$};
\filldraw[black] (15,-10) circle (2pt);
\filldraw[black] (15,-10) circle (0pt) node[below]{\footnotesize $\A_{7}^{2}$};

\draw[black, thick, ->] (11,-10) -- (12.8,-8.2);
\filldraw[black] (12,-9) circle (0pt) node[left]{\footnotesize $\h_{1,6}$};
\draw[black, thick, ->] (15,-10) -- (13.2,-8.2);
\filldraw[black] (14.5,-9.5) circle (0pt) node[left]{\footnotesize $\h_{2,6}$};

\filldraw[black] (13,-8) circle (2pt);
\filldraw[black] (13,-8) circle (0pt) node[above]{\footnotesize ${\A_{7}^{6}}^{-\{1,2\}}$};

\filldraw[black] (11,-8) circle (0pt) node{(11)};

\filldraw[black] (18,-10) circle (2pt);
\filldraw[black] (18,-10) circle (0pt) node[below]{\footnotesize $\A_{7}^{2}$};

\draw[black, thick, ->] (18,-10) -- (18,-8.2);
\filldraw[black] (18,-9) circle (0pt) node[right]{\footnotesize $\h_{2,6}$};

\filldraw[black] (18,-8) circle (2pt);
\filldraw[black] (18,-8) circle (0pt) node[above]{\footnotesize ${\A_{7}^{6}}^{-\{2\}}$};

\filldraw[black] (16,-8) circle (0pt) node{(12)};

\filldraw[black] (2,-16) circle (2pt);
\filldraw[black] (2,-16) circle (0pt) node[below]{\footnotesize $\A_{7}^{1}$};

\draw[black, thick, ->] (2,-16) -- (2,-14.2);
\filldraw[black] (2,-15) circle (0pt) node[right]{\footnotesize $\h_{1,6}$};

\filldraw[black] (2,-14) circle (2pt);
\filldraw[black] (2,-14) circle (0pt) node[above]{\footnotesize ${\A_{7}^{6}}^{-\{1\}}$};

\filldraw[black] (0,-14) circle (0pt) node{(13)};

\filldraw[black] (8,-16) circle (2pt);
\filldraw[black] (8,-16) circle (0pt) node[below]{\footnotesize $\A_{7}^{3}$};

\draw[black, thick, ->] (8,-16) -- (8,-14.2);
\filldraw[black] (8,-15) circle (0pt) node[right]{\footnotesize $\h_{3,6}$};

\filldraw[black] (8,-14) circle (2pt);
\filldraw[black] (8,-14) circle (0pt) node[above]{\footnotesize ${\A_{7}^{6}}^{-\{3\}}$};

\filldraw[black] (6,-14) circle (0pt) node{(14)};

\filldraw[black] (13,-16) circle (2pt);
\filldraw[black] (13,-16) circle (0pt) node[below]{\footnotesize $\A_{7}^{4}$};

\draw[black, thick, ->] (13,-16) -- (13,-14.2);
\filldraw[black] (13,-15) circle (0pt) node[right]{\footnotesize $\h_{4,6}$};

\filldraw[black] (13,-14) circle (2pt);
\filldraw[black] (13,-14) circle (0pt) node[above]{\footnotesize ${\A_{7}^{6}}^{-\{4\}}$};

\filldraw[black] (11,-14) circle (0pt) node{(15)};

\end{tikzpicture}
\caption{\label{A7-Spl-System}Systems of algebras isolated wrt  $\A_{7}$}
\end{figure}

\newpage

\section*{Summary}

In the article, starting from the given algebra, which is the Płonka sum, we presented a method for determining all its direct systems. Our method consists in decomposing the Płonka sums into the so-called isolated algebras, for which we can define the so-called Płonka homomorphisms. 
Our decomposition can be performed in two steps (cf. Section \ref{Section 5}). The first step consists of three components. First, we determine the family of algebras isolated wrt a given algebra. Secondly, we select all frames of isolated algebras. Thirdly, we modify the obtained frames by replacing the algebras of the given frame with complement algebras. In the second step, we select a family of P-homomorphisms. The structure obtained in this way constitutes a direct system of the initial algebra being a Płonka sum.
The presented decomposition method allows not only to determine all direct systems of a given Płonka sum but also check whether a given algebra is a Płonka sum.

In this paper, we focused on algebras with respect to which finitely many isolated algebras can be determined. An open problem is the question of the possibility of extending our result to algebras with respect to which we can determine infinitely many isolated algebras.
Additionally, in our work we did not analyze the problem of uniqueness of the decomposition of the Płonka sum into isolated algebras, in particular the problem of irreducibility of algebras that are elements of direct systems of the decomposed sum.
Although we already have partial results regarding these open problems, we will present a detailed analysis of them in subsequent works.


\newpage

\section*{Appendix}

\hypertarget{Appendix}{Examples \ref{Przyklad1} and \ref{Przyklad2} present Płonka sums; in turn, Example \ref{splituje,brak homomorfizmu} presents an algebra that is not a Płonka sum.} 

\begin{Przyklad}\label{Przyklad1}
Let us consider a join-semilattice $\I=\langle I,\leqslant\rangle$, where $I=\{1,2,3\}$, wherein $1, 2 \leqslant 3$, $1 \nleqslant 2$ and $2 \nleqslant 1$.
Let us consider three two-element Boolean algebras:
$\B_{1}=\langle\{a,b\}; \vee,\wedge,' \rangle$,  $\B_{2}=\langle\{c,d\}; \vee,\wedge,' \rangle$, $\B_3=\langle\{e,f\}; \vee,\wedge, '\rangle$.
We define homomorphisms $\h_{1,3}$ and $\h_{2,3}$ in the following way:
$\h_{1,3}(a)=e$, $\h_{1,3}(b)=f$ and $\h_{2,3}(c)=e$, $\h_{2,3}(d)=f$.

We can easily check that algebra $\A_1=\langle\{a,b,c,d,e,f\}; \vee,\wedge,' \rangle$, whose operations are defined in Table \ref{Def.oper.A1}, is the Płonka sum based on direct system $\langle\I,\{\B_{i}\}_{i\in I},\allowbreak\{\h_{i,j}\}_{i\leqslant j;i,j\in I}\rangle$: 
\begin{table}[ht!]
\centering
\special{em:linewidth 0.4pt}
\linethickness{0.5pt}
\scalebox{0.9}{
{
\begin{tabular}{||c||c|c|c|c|c|c||}
\hhline{|t:=:t:======:t|}
$\wedge$ & $a$ & $b$ & $c$ & $d$ & $e$ & $f$\\
\hhline{|:=::======:|}
$a$      & $a$ & $a$ & $e$ & $e$ & $e$ & $e$ \\
\hhline{||-||-|-|-|-|-|-||}
$b$      & $a$ & $b$ & $e$ & $f$ & $e$ & $f$ \\
\hhline{||-||-|-|-|-|-|-||}
$c$      & $e$ & $e$ & $c$ & $c$ & $e$ & $e$ \\
\hhline{||-||-|-|-|-|-|-||}
$d$      & $e$ & $f$ & $c$ & $d$ & $e$ & $f$ \\
\hhline{||-||-|-|-|-|-|-||}
$e$      & $e$ & $e$ & $e$ & $e$ & $e$ & $e$ \\
\hhline{||-||-|-|-|-|-|-||}
$f$      & $e$ & $f$ & $e$ & $f$ & $e$ & $f$ \\
\hhline{|b:=:b:======:b|}
\end{tabular}
\hspace{0.4cm}
\begin{tabular}{||c||c|c|c|c|c|c||}
\hhline{|t:=:t:======:t|}
$\vee$ & $a$ & $b$ & $c$ & $d$ & $e$ & $f$\\
\hhline{|:=::======:|}
$a$      & $a$ & $b$ & $e$ & $f$ & $e$ & $f$ \\
\hhline{||-||-|-|-|-|-|-||}
$b$      & $b$ & $b$ & $f$ & $f$ & $f$ & $f$ \\
\hhline{||-||-|-|-|-|-|-||}
$c$      & $e$ & $f$ & $c$ & $d$ & $e$ & $f$ \\
\hhline{||-||-|-|-|-|-|-||}
$d$      & $f$ & $f$ & $d$ & $d$ & $f$ & $f$ \\
\hhline{||-||-|-|-|-|-|-||}
$e$      & $e$ & $f$ & $e$ & $f$ & $e$ & $f$ \\
\hhline{||-||-|-|-|-|-|-||}
$f$      & $f$ & $f$ & $f$ & $f$ & $f$ & $f$ \\
\hhline{|b:=:b:======:b|}
\end{tabular}
\hspace{0.4cm}
\begin{tabular}{||c||c||}
\hhline{|t:=:t:=:t|}
$x$ & $x'$ \\
\hhline{|:=::=:|}
$a$      & $b$ \\
\hhline{||-||-||}
$b$      & $a$ \\
\hhline{||-||-||}
$c$      & $d$\\
\hhline{||-||-||}
$d$      & $c$\\
\hhline{||-||-||}
$e$      & $f$\\
\hhline{||-||-||}
$f$      & $e$\\
\hhline{|b:=:b:=:b|}
\end{tabular}
}}
\caption{\label{Def.oper.A1}Operations of algebra $\A_{1}$}
\end{table}
\end{Przyklad}


\begin{Przyklad}\label{Przyklad2}
Let us modify Example \ref{Przyklad1}. Let us consider a join-semilattice $\I=\langle I,\leqslant\rangle$, where  $I=\{1,2,3\}$, 
wherein $1 \leqslant 2 \leqslant 3$ and --- as in Example \ref{Przyklad1} --- let us consider three two-element Boolean algebras:
$\B_{1}$, $\B_{2}$, $\B_{3}$. 
We also consider homomorphisms $\h_{1,2}:B_{1}\fun B_{2}$ and $\h_{2,3}:B_{2}\fun B_{3}$ defined in the following way:
$\h_{1,2}(a)=c$, $\h_{1,2}(b)=d$ and $\h_{2,3}(c)=e$, $\h_{2,3}(d)=f$.

Algebra $\A_2=\langle\{a,b,c,d,e,f\}; \vee,\wedge,' \rangle$, whose operations are defined in Table \ref{Def.oper.A2}, is a Płonka sum based on direct system $\langle\I,\{\B_{i}\}_{i\in I},\allowbreak\{\h_{i,j}\}_{i\leqslant j;i,j\in I}\rangle$:   
\begin{table}[ht!]
\centering
\special{em:linewidth 0.4pt}
\linethickness{0.5pt}
\scalebox{0.9}{
{
\begin{tabular}{||c||c|c|c|c|c|c||}
\hhline{|t:=:t:======:t|}
$\wedge$ & $a$ & $b$ & $c$ & $d$ & $e$ & $f$\\
\hhline{|:=::======:|}
$a$      & $a$ & $a$ & $c$ & $c$ & $e$ & $e$ \\
\hhline{||-||-|-|-|-|-|-||}
$b$      & $a$ & $b$ & $c$ & $c$ & $e$ & $f$ \\
\hhline{||-||-|-|-|-|-|-||}
$c$      & $c$ & $c$ & $c$ & $c$ & $e$ & $e$ \\
\hhline{||-||-|-|-|-|-|-||}
$d$      & $c$ & $d$ & $c$ & $d$ & $e$ & $f$ \\
\hhline{||-||-|-|-|-|-|-||}
$e$      & $e$ & $e$ & $e$ & $e$ & $e$ & $e$ \\
\hhline{||-||-|-|-|-|-|-||}
$f$      & $e$ & $f$ & $e$ & $f$ & $e$ & $f$ \\
\hhline{|b:=:b:======:b|}
\end{tabular}
\hspace{0.4cm}
\begin{tabular}{||c||c|c|c|c|c|c||}
\hhline{|t:=:t:======:t|}
$\vee$ & $a$ & $b$ & $c$ & $d$ & $e$ & $f$\\
\hhline{|:=::======:|}
$a$      & $a$ & $b$ & $c$ & $d$ & $e$ & $f$ \\
\hhline{||-||-|-|-|-|-|-||}
$b$      & $b$ & $b$ & $d$ & $d$ & $f$ & $f$ \\
\hhline{||-||-|-|-|-|-|-||}
$c$      & $c$ & $d$ & $c$ & $d$ & $e$ & $f$ \\
\hhline{||-||-|-|-|-|-|-||}
$d$      & $d$ & $d$ & $d$ & $d$ & $f$ & $f$ \\
\hhline{||-||-|-|-|-|-|-||}
$e$      & $e$ & $f$ & $e$ & $f$ & $e$ & $f$ \\
\hhline{||-||-|-|-|-|-|-||}
$f$      & $f$ & $f$ & $f$ & $f$ & $f$ & $f$ \\
\hhline{|b:=:b:======:b|}
\end{tabular}
\hspace{0.4cm}
\begin{tabular}{||c||c||}
\hhline{|t:=:t:=:t|}
$x$ & $x'$ \\
\hhline{|:=::=:|}
$a$      & $b$ \\
\hhline{||-||-||}
$b$      & $a$ \\
\hhline{||-||-||}
$c$      & $d$\\
\hhline{||-||-||}
$d$      & $c$\\
\hhline{||-||-||}
$e$      & $f$\\
\hhline{||-||-||}
$f$      & $e$\\
\hhline{|b:=:b:=:b|}
\end{tabular}
}}
\caption{\label{Def.oper.A2}Operations of algebra $\A_{2}$}
\end{table}
\end{Przyklad}

\begin{Przyklad}\label{splituje,brak homomorfizmu}
Now let the join-semilattice be the structure $\I=\langle I,\leqslant\rangle$, where $I=\{1,2\}$,
with $1 \leqslant 2 $ and --- as in Examples \ref{Przyklad1} and \ref{Przyklad2} --- two two-element Boolean algebras will be given:
$\B_{1}$, $\B_{2}$. Consider function $\h_{1,2}:B_{1}\fun B_{2}$ defined in the following way:
$\h_{1,2}(a)=c$, $\h_{1,2}(b)=c$.

Algebra $\A_3=\langle\{a,b,c,d,\}; \vee,\wedge,'\rangle$, whose operations are defined in Table \ref{Def.oper.A3}, is not a Płonka sum. 
\begin{figure}[ht!]
\centering
\special{em:linewidth 0.4pt}
\linethickness{0.5pt}
\scalebox{1}{
{
\begin{tabular}{||c||c|c|c|c||}
\hhline{|t:=:t:====:t|}
$\wedge$ & $a$ & $b$ & $c$ & $d$ \\
\hhline{|:=::====:|}
$a$      & $a$ & $a$ & $c$ & $c$  \\
\hhline{||-||-|-|-|-||}
$b$      & $a$ & $b$ & $c$ & $c$  \\
\hhline{||-||-|-|-|-||}
$c$      & $c$ & $c$ & $c$ & $c$  \\
\hhline{||-||-|-|-|-||}
$d$      & $c$ & $c$ & $c$ & $d$  \\
\hhline{|b:=:b:====:b|}
\end{tabular}
\hspace{0.4cm}
\begin{tabular}{||c||c|c|c|c||}
\hhline{|t:=:t:====:t|}
$\vee$   & $a$ & $b$ & $c$ & $d$ \\
\hhline{|:=::====:|}
$a$      & $a$ & $b$ & $c$ & $d$ \\
\hhline{||-||-|-|-|-||}
$b$      & $b$ & $b$ & $c$ & $d$ \\
\hhline{||-||-|-|-|-||}
$c$      & $c$ & $c$ & $c$ & $d$ \\
\hhline{||-||-|-|-|-||}
$d$      & $d$ & $d$ & $d$ & $d$ \\
\hhline{|b:=:b:====:b|}
\end{tabular}
\hspace{0.4cm}
\begin{tabular}{||c||c||}
\hhline{|t:=:t:=:t|}
$x$ & $x'$ \\
\hhline{|:=::=:|}
$a$      & $b$ \\
\hhline{||-||-||}
$b$      & $a$ \\
\hhline{||-||-||}
$c$      & $d$\\
\hhline{||-||-||}
$d$      & $c$\\
\hhline{|b:=:b:=:b|}
\end{tabular}
}}
\caption{\label{Def.oper.A3}Operations of algebra $\A_{3}$}
\end{figure}

We note that $\h_{1,2}$ is not a homomorphism. For example, we have $\h_{1,2}(a')\allowbreak=c\neq d=\h_{1,2}(a)'$. It is known that in general we have two functions from $\B_{1}$ to $\B_{2}$ that could be taken into account to show that $\A_{3}$ is a Płonka sum. The function $\psi\colon B_{1}\fun B_{2}$ such that $\psi(a)=c$ and $\psi(b)=d$ is a homomorphism in the analyzed case, but does not meet the condition of operation of a Płonka sum (see Definition \ref{def:suma-Plonki}). For $\A_{3}$ we have $b\vee c=c$. In turn, the function $\chi\colon B_{1}\fun B_{2}$ such that $\chi(a)=d$ and $\chi(b)=c$ will not be a homomorphism in the analyzed case. For $\A_{3}$ we have $\chi(a\wedge b)=d\neq c=\chi(a)\wedge\chi(b)$.
\end{Przyklad}  

In Example \ref{Przyklad-Spl} we present isolated algebras specified wrt algebras consider in the examples presented above.

\begin{Przyklad}\label{Przyklad-Spl}
\begin{itemize}
\item With respect to algebra $\A_1$, isolated algebras are subalgebras of $\A_{1}$ with universes: $\{a,b\}$, $\{c,d\}$ and $\{a,b,c,d,e,f\}$. The algebra whose universe is $\{e,f\}$ is not isolated wrt $\A_{1}$.

\item With respect to algebra $\A_2$, isolated algebras are subalgebras of $\A_{2}$ with universes: $\{a,b\}$, $\{a,b,c,d\}$
and $\{a,b,c,d,e,f\}$. Algebras with universes $\{c,d\}$ and $\{e,f\}$ are not isolated wrt $\A_{2}$.

\item With respect to the algebra $\A_3$, isolated algebras are subalgebras of $\A_{3}$ with the universes: $\{a,b\}$ and $\{a,b,c,d\}$. 
\end{itemize}
\end{Przyklad}

In Example \ref{Przyklad-Nieskonczona-SP} we presented an algebra with respect to which we can determine finitely many infinite isolated algebras.

\begin{Przyklad}\label{Przyklad-Nieskonczona-SP}
Let $\mathbb{O}$ be the set of odd natural numbers and $\mathbb{E}$ be the set of even numbers without zero. Let us consider a Płonka sum $\A=\langle \mathbb{N};\ast\rangle$, where $\mathbb{N}$ is the set of natural numbers without zero and $\ast$ is a binary operation defined in the following way:
\[
n\ast m=\begin{cases}
1,\quad & \text{if }n,m\in\mathbb{O}\\
2,\quad & \text{otherwise}.
\end{cases}
\]
Two algebras are isolated wrt $\A$: $\A_{1}=\langle\mathbb{O};\ast\rangle$ and $\A$. Note that algebra $\A_{2}:=\langle\mathbb{E};\ast\rangle=\langle\mathbb{N}\setminus\mathbb{O};\ast\rangle$ is not isolated wrt $\A$. Function $\varphi_{1,2}:\mathbb{O}\fun\mathbb{E}$ s.t. $\h_{1,2}(n)=n+1$
is a homomorphism.
The Płonka sum $\A$ can be determined by direct system $
\langle\langle\{1,2\},\leq\rangle,\{\A_{1},\A_{2}\},\{\h_{1,1},\h_{1,2},\h_{2,2}\}\rangle$.
\end{Przyklad}

Example \ref{Przyklad4} shows that from a given complement algebra to a complement algebra of a `greater' index a homomorphism does not always exist.

\begin{Przyklad}\label{Przyklad4}
Let us consider algebra $\A_{4}=\langle\{a,b,c,d\};\ast\rangle$, where $\ast$ is defined in Table \ref{Def.oper.A4}. 
\begin{table}[ht!]
\centering
\special{em:linewidth 0.4pt}
\linethickness{0.5pt}
\begin{tabular}{||c||c|c|c|c||}
\hhline{|t:=:t:====:t|}
$\ast$   & $a$ & $b$ & $c$ & $d$ \\
\hhline{|:=::====:|}
$a$      & $a$ & $b$ & $c$ & $d$ \\
\hhline{||-||-|-|-|-||}
$b$      & $b$ & $b$ & $c$ & $d$ \\
\hhline{||-||-|-|-|-||}
$c$      & $c$ & $c$ & $d$ & $c$ \\
\hhline{||-||-|-|-|-||}
$d$      & $d$ & $d$ & $d$ & $c$ \\
\hhline{|b:=:b:====:b|}
\end{tabular}
\caption{\label{Def.oper.A4}Operations of algebra $\A_{4}$}
\end{table}

Three algebras are isolated wrt $\A_{4}$: $\B_{1}=\langle\{a\};\ast\rangle$, $\B_{2}=\langle\{a,b\};\ast\rangle$ and $\B_{3}=\langle\{a,b,c,d\};\ast\rangle$. There is no homomorphism from $\B_{2}$ to $\B_{3}\setminus\B_{2}$. However, algebra $\A_{4}$ is a Płonka sum and can be determined by the following non-trivial direct system $\langle\langle\{1,2\},\{\langle 1,1\rangle,\langle 2,2\rangle,\langle 1,2\rangle\}\rangle,\{\C_{1},\allowbreak\C_{2}\},\{\varphi_{1,1},\varphi_{2,2},\varphi_{1,2}\}\rangle$, 
where $\dom(\C_{1})=\{a\}$, $\dom(\C_{2})=\{b,c,d\}$ and $\varphi_{1,2}(a)=b$.
\end{Przyklad}

Example \ref{Przyklad5} shows that not every homomorphism from a complement algebra to a complement algebra is a homomorphism suitable for determining the Płonka sum operation in the manner specified in  Definition \ref{def:suma-Plonki}. 

\begin{Przyklad}\label{Przyklad5}
Let us consider a Płonka sum $\A_{5}=\langle\{a,b,c,d\};\ast\rangle$, where $\ast$ is defined in Table \ref{Def.oper.A5}.
\begin{table}[!h]
\centering
\special{em:linewidth 0.4pt}
\linethickness{0.5pt}
\begin{tabular}{||c||c|c|c|c||}
\hhline{|t:=:t:====:t|}
$\ast$   & $a$ & $b$ & $c$ & $d$ \\
\hhline{|:=::====:|}
$a$      & $a$ & $c$ & $b$ & $a$ \\
\hhline{||-||-|-|-|-||}
$b$      & $c$ & $b$ & $c$ & $c$ \\
\hhline{||-||-|-|-|-||}
$c$      & $b$ & $b$ & $c$ & $b$ \\
\hhline{||-||-|-|-|-||}
$d$      & $a$ & $c$ & $b$ & $d$ \\
\hhline{|b:=:b:====:b|}
\end{tabular}
\caption{\label{Def.oper.A5}Operation $\ast$ of algebra $\A_{5}$}
\end{table}
Płonka sum $\A_{5}$ can be determined by direct system
$\langle\langle\{1,2\},\{\langle 1,1\rangle,\langle 1,2\rangle,\langle 2,2\rangle\}\rangle,\{\C_{1},\C_{2}\},\{\varphi_{1,1},\varphi_{1,2},\varphi_{2,2}\}\rangle$,
where $\dom(\C_{1})=\{d\}$, $\dom(\C_{2})=\{a,b,c\}$ and $\varphi_{1,2}(d)=a$.

Three algebras are isolated wrt $\A_{5}$: $\B_{1}=\langle\{d\};\ast\rangle$, $\B_{2}=\langle\{a,d\};\ast\rangle$ and $\B_{3}=\langle\{a,b,c,\allowbreak d\};\ast\rangle$. Note that there is a homomorphism from $\B_{2}^{-}$ to $\B_{3}^{-}$ (we put: $\varphi(a)=b$). However, this homomorphism cannot be taken into account when trying to determine the direct system of $\A_{5}$. 
We can see that $a\ast c=b\neq c= b\ast c=\varphi(a)\ast c$. Homomorphism $\varphi$ does not satisfy the condition of Definition \ref{def:suma-Plonki} for operation of $\A_{5}$.  
\end{Przyklad}

In Example \ref{Przyklad-P-hom}, we present several P-homomorphisms of the algebras discussed in the examples given above.

\begin{Przyklad}\label{Przyklad-P-hom}
To indicate example of P-homomorphisms, let us return to algebra $\A_{1}$ from Example \ref{Przyklad1} and algebra $\A_{2}$ from Example \ref{Przyklad2}.
\begin{itemize}
\item We have the following P-homomorphisms wrt $\A_{1}$: $\varphi_{1,3}$ from algebra $\B_{1}$ to algebra $\B_{3}$ and $\varphi_{2,3} $ from algebra $\B_{2}$ to algebra $\B_{3}$. 

\item We have the following P-homomorphisms wrt $\A_{2}$: $\varphi_{1,2}$ from algebra $\B_{1}$ to algebra $\B_{2}$ and $\varphi_{2,3}$ from algebra $\B_{2}$ to algebra $\B_{3}$. 
\end{itemize}
In the case of algebra $\A_{5}$ from Example \ref{Przyklad5}, we indicated a homomorphism that is not a P-homomorphism wrt $\A_{5}$.
\begin{itemize}
\item Homomorphism $\h$ from $\B_{2}^{-}$ to $\B_{3}^{-}$ is not a P-homomorphism wrt $\A_5$. We noticed that $a\ast c\neq \varphi(a)\ast c$. 
\end{itemize}
\end{Przyklad}

Example \ref{Przyk.:P-homomorfizmy} shows that there might be more than one P-homomorphism from complement algebras to other complement algebras.    

\begin{Przyklad}\label{Przyk.:P-homomorfizmy}
Let us consider Płonka sum $\A_{6}=\langle\{a,b,c,d,e,f\};\ast\rangle$, where binary operation $\ast$ is defined in Table \ref{Def.oper.A6}.
\begin{table}[h!]
\centering
\special{em:linewidth 0.4pt}
\linethickness{0.5pt}
\begin{tabular}{||c||c|c|c|c|c|c||}
\hhline{|t:=:t:======:t|}
$\ast$ & $a$ & $b$ & $c$ & $d$ & $e$ & $f$\\
\hhline{|:=::======:|}
$a$      & $a$ & $a$ & $c$ & $c$ & $e$ & $e$ \\
\hhline{||-||-|-|-|-|-|-||}
$b$      & $a$ & $a$ & $c$ & $c$ & $e$ & $e$ \\
\hhline{||-||-|-|-|-|-|-||}
$c$      & $c$ & $c$ & $c$ & $c$ & $e$ & $e$ \\
\hhline{||-||-|-|-|-|-|-||}
$d$      & $c$ & $c$ & $c$ & $c$ &$e$& $e$ \\
\hhline{||-||-|-|-|-|-|-||}
$e$      & $e$ & $e$ & $e$ & $e$ & $e$ & $e$ \\
\hhline{||-||-|-|-|-|-|-||}
$f$      & $e$ & $e$ & $e$ & $e$ & $e$ & $e$ \\
\hhline{|b:=:b:======:b|}
\end{tabular}
\caption{\label{Def.oper.A6}Operation $\ast$ of algebra $\A_{6}$}
\end{table}

Three algebras are isolated wrt $\A_{6}$: $\B_{1}=\langle\{a,b\};\ast\rangle$, $\B_{2}=\langle\{a,b,c,d\};\ast\rangle$ and of course $\B_{3}=\A_{6}$. Note that there are two homomorphisms from $\B_{1}^{-}=\B_{1}$ to $\B_{2}^{-}$ and two homomorphisms from $\B_{2}$ to $\B_{3}^{-}$, namely:
\begin{itemize}
\item $\varphi_{1,2}: B_1^{-}\longrightarrow B_{2}^{-}$: $\varphi_{1,2}(a)=c$ and $\varphi_{1,2}(b)=d$,
\item $\psi_{1,2}:B_1^{-}\longrightarrow B_{2}^{-}$: $\psi_{1,2}(a)=\psi_{1,2}(b)=c$,
\item $\varphi_{2,3}: B_2^{-}\longrightarrow B_{3}^{-}$:  $\varphi_{2,3}(c)=e$ and $\varphi_{2,3}(d)=f$,
\item $\psi_{2,3}: B_2^{-}\longrightarrow B_{3}^{-}$: $\psi_{2,3}(c)=\psi_{2,3}(d)=e$.
\end{itemize}

It is easy to check that each of these functions is a P-homomorphism. So, we have one frame of isolated algebras and four possible choices of P-homomorphisms. 
\end{Przyklad}

\end{document}